\documentclass[12pt]{amsart}
\usepackage{amscd,amsmath,amssymb,amsfonts}
\usepackage[cmtip, all]{xy}

\newlength{\rulebreite}

\def\timesover#1#2#3{\ \xymatrix@1@=0pt@M=0pt{ _{#1}&\times&_{#2} \\& ^{#3}&}\ }
\def\otimesover#1#2#3{\ \xymatrix@1@=0pt@M=0pt{ _{#1}&\otimes&_{#2} \\& ^{#3}&}\ }

\theoremstyle{plain}
\newtheorem{thm}{Theorem}
\newtheorem{lem}[thm]{Lemma}
\newtheorem{cor}[thm]{Corollary}
\newtheorem{prop}[thm]{Proposition}
\theoremstyle{definition}
\newtheorem{defn}[thm]{Definition}
\newtheorem{conj}[thm]{Conjecture}

\newtheorem{question}[thm]{Question}
\newtheorem{rmk}[thm]{Remark}

\newtheorem{claim}[thm]{Claim}
\numberwithin{thm}{section}
\numberwithin{equation}{section}

\newcommand{\eq}[2]{\begin{equation}\label{#1}#2 \end{equation}}
\newcommand{\ml}[2]{\begin{multline}\label{#1}#2 \end{multline}}
\newcommand{\ga}[2]{
\begin{gather}\label{#1}#2\end{gather} 
}

\newcommand{\rank}{{\rm rank}}

\newcommand{\Div}{{\rm Div}}
\newcommand{\Hom}{{\rm Hom}}
\newcommand{\im}{{\rm im}}
\newcommand{\Spec}{{\rm Spec \,}}

\newcommand{\Gal}{{\rm Gal}}

\newcommand{\sC}{{\mathcal C}}

\newcommand{\sE}{{\mathcal E}}

\newcommand{\sI}{{\mathcal I}}

\newcommand{\sL}{{\mathcal L}}

\newcommand{\sO}{{\mathcal O}}
\newcommand{\sP}{{\mathcal P}}

\newcommand{\sR}{{\mathcal R}}

\newcommand{\sV}{{\mathcal V}}
\newcommand{\sW}{{\mathcal W}}

\newcommand{\dR}{{\mathcal{R}}}
\newcommand{\dV}{{\mathcal{V}}}

\newcommand{\A}{{\mathbb A}}

\newcommand{\C}{{\mathbb C}}

\newcommand{\F}{{\mathbb F}}
\newcommand{\G}{{\mathbb G}}

\newcommand{\N}{{\mathbb N}}
\renewcommand{\P}{{\mathbb P}}
\newcommand{\Q}{{\mathbb Q}}
\newcommand{\R}{{\mathbb R}}

\newcommand{\Z}{{\mathbb Z}}

\newcommand{\Cu}{{\rm Cu}}
\newcommand{\Sm}{{\rm Sm}}

\newcommand{\GL}{{\rm GL}}
\newcommand{\Ql}{{\bar \Q_\ell}}
\newcommand{\Sw}{{\rm Sw}}
\newcommand{\Dis}{{\rm D}}
\newcommand{\CH}{{\rm CH}}

\def\tilde{\widetilde}

\begin{document}

\title[Finiteness theorem (after Deligne)]{A finiteness theorem for Galois representations of function fields over finite fields  (after Deligne)}
\author{H\'el\`ene Esnault}
\address{
Universit\"at Duisburg-Essen, Mathematik, 45117 Essen, Germany}
\email{esnault@uni-due.de}
\author{Moritz Kerz}
\address{Universit\"at Regensburg, Fakult\"at f\"ur Mathematik, 93040
Regensburg, Germany}
\email{moritz.kerz@mathematik.uni-regensburg.de}
\date{ September 13, 2012}
\thanks{The first author is supported by  the SFB/TR45 and
the ERC Advanced Grant 226257, the second author by the DFG Emmy Noether-Nachwuchsgruppe ``Arithmetik \"uber endlich erzeugten K\"orpern''}
\begin{abstract}

We give a detailed account of Deligne's  letter  \cite{DelFinitude} to  Drinfeld dated 
June 18, 2011, in which
 he shows that 
 there are finitely many    irreducible 
  lisse $\bar \Q_\ell$-sheaves with bounded  ramification, up to isomorphism
  and up to twist, on a  smooth variety defined over a finite field. The proof relies on
Lafforgue's Langlands correspondence over curves \cite{Laf}. In addition, Deligne shows
the existence of affine moduli of finite type over $\mathbb{Q}$.
A corollary of Deligne's finiteness theorem is the existence of a number field which contains all traces of the Frobenii at closed points, 
which was the main result of  \cite{DelNumberField}  and which answers positively  his own conjecture
   \cite[Conj.~1.2.10~(ii)]{WeilII}.

\end{abstract}
\maketitle

\section{Introduction } \label{intro}
\noindent
In Weil II   \cite[Conj.~1.2.10]{WeilII} Deligne  conjectured that if $X$ is a normal connected scheme of finite type over a finite field of characteristic $p$, and $V$ is an irreducible  lisse 
$\bar{\Q}_\ell$-sheaf of rank $r$,  with finite determinant, then 
\begin{itemize}
 \item[(i)] $V$ has weight $0$, 
\item[(ii)] there is a number field $E(V)\subset \bar{ \Q}_\ell$
containing all the coefficients of the local characteristic  polynomials ${\rm
  det}(1-tF_x|_{V_x})$, where $x$ runs through the closed points of $X$ and $F_x$ is the
geometric Frobenius at the point $x$, 
\item[(iii)] $V$ admits $\ell'$-companions for all prime
numbers $\ell'\neq p$.  
\end{itemize}
As an application of his Langlands correspondence for ${\rm GL}_r$, 
Lafforgue \cite{Laf} proved  (i), (ii), (iii) for $X$ a smooth curve, out of which one
deduces (i) in general. Using Lafforgue's results, Deligne showed (ii) in
\cite{DelNumberField}. Using (ii) and ideas of Wiesend, Drinfeld \cite{Drinfeld} showed
(iii) assuming in addition $X$ to be smooth. A slightly more elementary variant of
Deligne's argument for (ii) was given in \cite{EK}.

\medskip

Those conjectures were formulated with the hope that a more motivic statement could be
true,  which would say that those lisse sheaves come from geometry.  On the other hand,
over smooth varieties over the field of complex numbers, Deligne  in \cite{DelMos} showed
finiteness of  $\Q$-summands of polarized variations of pure Hodge structures over $\Z$ of bounded rank, a theorem
which, in weight one, is due to Faltings \cite{Fa}. Those are always regular
singular,
while lisse $\bar{\Q}_\ell$-sheaves are not necessarily tame. However,  any lisse sheaf has
bounded ramification (see  Proposition~\ref{prop.rambound} for details). Furthermore, one may  twist a lisse  $\bar{\Q}_\ell$-sheaf by a character coming from the ground field. Thus
it is natural to expect:

\begin{thm}[Deligne]
There are only finitely many irreducible lisse $\bar{\Q}_\ell$-sheaves of given rank up to twist on $X$ with
suitably bounded ramification at infinity.
\end{thm}
 Deligne shows this theorem in 
\cite{DelFinitude} by extending his arguments from
\cite{DelNumberField}. A precise formulation is given in Theorem~\ref{thm.finite1} based
on the
ramification theory explained in Section~\ref{sec.gloram}.

Our aim in this note is to give a detailed account of Deligne's proof of this finiteness
theorem for lisse $\Ql$-sheaves and consequently of his proof of (ii). For some remarks on the
difference between our method  and Deligne's original argument for
proving (ii) in
\cite{DelNumberField} see Section~\ref{sec.coeff}.

\medskip 

In fact Deligne shows a stronger finiteness theorem which comprises finiteness of the
number of what we call {\em 2-skeleton sheaves} \footnote{we thank Lars Kindler for suggesting this terminology} on $X$. A 2-skeleton sheaf consists of an
isomorphism class of a
semi-simple lisse $\Ql$-sheaf on every smooth curve mapping to $X$, which are assumed to be compatible in
a suitable sense. 
These 2-skeleton sheaves were first studied by Drinfeld \cite{Drinfeld}. His main theorem
roughly says that if a 2-skeleton sheaf is tame at infinity along each curve then it comes
from a lisse sheaf on $X$, extending the rank one case treated in \cite{SchSp}, \cite{W}.
Deligne suggests that a more general statement should be true:

\begin{question}
Does any 2-skeleton sheaf with bounded ramification come from a lisse $\Ql$-sheaf on $X$?
\end{question}

For a precise formulation of the question see Question~\ref{fund.quest}. The
answer to this question is not even known for rank one sheaves, in which case the problem has
been suggested already earlier in higher dimensional class field theory. On the other hand
Deligne's finiteness for 2-skeleton sheaves has interesting consequences for relative Chow
groups of $0$-cycles over finite fields, see Section~\ref{sec.finchow}.


\medskip

Some comments on the proof of the finiteness theorem:
Deligne uses in a crucial way his key theorem \cite[Prop.~2.5]{DelNumberField} on curves asserting that
a semi-simple lisse $\bar{\Q}_\ell$-sheaf is uniquely determined by its 
characteristic polynomials of the Frobenii at all closed points of some explicitly bounded
degree, see Theorem~\ref{frob.thm}.
This enables him to construct a coarse moduli space of 2-skeleton sheaves $L_r(X,D)$  as an affine
scheme of finite type over $\Q$, such that its $\bar{\Q}_\ell$-points correspond to the
2-skeleton sheaves of rank $r$ and bounded ramification by the given divisor $D$ at infinity.
 
We simplify Deligne's construction of the moduli space slightly. Our method yields
less information on the resulting moduli, yet it is enough to deduce the finiteness
theorem. In fact finiteness is seen by showing that irreducible lisse
$\bar{\Q}_\ell$-sheaves up to twist are in bijection with (some of) the one-dimensional irreducible
components of the moduli space (Corollary~\ref{cor.irr}). 

\medskip

We give some applications of Deligne's finiteness theorem in Section~\ref{appl}.

Firstly, it implies the existence of a number field $E(V)$ as in (ii) above, 
see
Theorem~\ref{mainthmirr}. This number field is in fact stable by an ample hyperplane section
if $X$ is projective, see Proposition~\ref{hyperplane}.

 Secondly, as mentioned above the degree zero part of the relative Chow group of
$0$-cycles with  bounded modulus is finite (Theorem~\ref{CH}).

\medskip

Deligne addresses the question of the number of irreducible lisse $\bar{\Q}_\ell$-sheaves with bounded ramification. In \cite{DF} some concrete examples on the projective line minus a divisor of degree $\le 4$ are computed.
In Section~\ref{conjecture} we formulate Deligne's  qualitative conjecture. This formulation rests on emails he sent us and on his lecture in June 2012 in Orsay on the occasion of the Laumon conference.

\medskip

{\it Acknowledgment:} 
Our note gives an account of the $9$  dense pages  written by Deligne to Drinfeld
\cite{DelFinitude}. They rely on \cite{DelNumberField} and \cite{Drinfeld} and contain a completely new idea of great beauty, to the effect of  showing  finiteness by constructing moduli of finite type and 
 equating the classes of the sheaves one wants to count with some of
the irreducible components.
 We thank Pierre Deligne for his willingness to read our note and
for his many enlightening comments.

 Parts of the present note are taken from  our seminar note \cite{EK}. They grew
out of discussions at the Forschungsseminar at Essen
 during summer 2011. We thank all participants of the seminar.

 We thank Ng\^o Bao Ch\^au for a careful reading of our article and several comments which contributed to improve its presentation, we thank him and and Ph\`ung H\^o Hai for giving
us the possibility to publish this note on the occasion of the first VIASM
Annual Meeting.

\section{The finiteness theorem and some consequences} \label{sec.intro}

\subsection{Deligne's finiteness theorem for sheaves} \label{s1.1}

We begin by formulating a version of Deligne's finiteness theorem for $\ell$-adic Galois
representations of functions fields. Later in this section we
introduce the  notion of a 2-skeleton  $\ell$-adic representation, which is necessary in
order to state a  stronger form of Deligne's finiteness result. 

\medskip

Let $\Sm_{\F_q}$ be the category of smooth separated schemes  $X/\F_q$ of finite type over the finite field $\F_q$. 
We fix once for all an algebraic closure $\F\supset \F_q$.
To
$X\in \Sm_{\F_q}$ connected one associates functorially the {\em Weil group} $W(X)$ \cite[~1.1.7 ]{WeilII},  a
topological group, well-defined up to an inner automorphism by $\pi_1(X\otimes_{\F_q} \F )$ when $X$ is geometrically  connected  over $\F_q$. If so, 
then it
sits in an exact sequence
\[
0\to \pi_1(X \otimes_{\F_q} \F ) \to  W(X) \to  W(\F_q)  \to 0.
\]
 There is a canonical identification $W(\F_q)=\Z$.

We fix a prime number $\ell$ with $(\ell,q)=1$.
Let $\dR_r(X)$ be the set of lisse $\bar \Q_\ell$-Weil sheaves on $X$ of dimension $r$ up to isomorphism and up to semi-simplification. 
For $X$ connected,
a lisse $\bar \Q_\ell$-Weil sheaf on $X$ is the same as a continuous representations
$W(X) \to {\rm GL}_r(\bar \Q_\ell )$. As we do not want to talk about a topology on $\Ql$
we define the latter continuous representations ad hoc as the homomorphisms which factor through a continuous homomorphism $W(X) \to {\rm GL}_r( E )$ for some finite
extension $E$ of $\Q_\ell$, see \cite[(1.1.6)]{WeilII}.

The weaker form of the finiteness theorem says that the number of  classes of irreducible sheaves in $\dR_r(X)$ with bounded wild
ramification is finite up to twist.
Let us give some more details. Let $X\subset \bar X$ be a normal compactification of the
connected scheme $X$
such that $\bar X \setminus X$ is the support of an effective Cartier divisor on $\bar X$. 
Let $D\in \Div^+( \bar X)$ be an effective Cartier divisor with support in $\bar X
\setminus X$. In Section~\ref{sec.gloram} we will define a subset $\dR_r(X,D)$ of representations whose
Swan conductor along any smooth curve mapping to $\bar X$  is bounded by the pullback of
$D$ to the completed curve. We show that for
any $V\in \dR_r(X)$ there is a divisor $D$ with $V\in \dR_r(X,D)$, see Proposition~\ref{prop.rambound}.

For $V\in \dR_r(X,D)$ we have the notion of twist $\chi \cdot V$ by an element $\chi\in \dR_1(\F_q)$.  

\begin{thm}[Deligne] \label{thm.finite1} 
 Let $X\in \Sm_{\F_q}$ be  connected and $D\in \Div^+(\bar X)$ be an effective Cartier divisor with 
  support in $\bar X \setminus X$.  The 
 set of irreducible sheaves $V\in \dR_r(X,D)$ is finite up to twist by elements of $\dR_1(\F_q)$. 
\end{thm}

In particular the theorem implies that for any integer $N>0$ there are only finitely many  irreducible $V \in
\dR_r(X,D)$ with $\det(V)^{\otimes N} =1$. 
 Theorem~\ref{thm.finite1} is a consequence of
the stronger Finiteness Theorem~\ref{thm.finite2}. 


\begin{rmk}\label{weil.etale}
 Any irreducible lisse Weil sheaf on $X$ is a twist of an \'etale sheaf, Proposition~\ref{weil.etale.gen}. So the theorem could also be stated  with \'etale sheaves instead of Weil sheaves.

\end{rmk}

\subsection{Existence problem and  finiteness theorem for 2-skeleton sheaves} \label{ss.strong}

By ${\Cu}(X)$ we denote the set of normalizations of closed integral
subschemes of $X$ of dimension one. 

We say that a family $(V_C)_{C\in \Cu(X)}$ with $V_C\in \dR_r(C)$ is {\em compatible} if for  all pairs $(C, 
 C')$ we have 
 \eq{}{V_C|_{(C \times_X C')_{\rm red}  } = V_{C'}|_{(C \times_X C')_{\rm red} }
\in \dR_r( (C \times_X C')_{\rm red}  ) . \notag}
 We write $\dV_r(X)$ for the set of compatible families -- also called
{\em 2-skeleton sheaves}.

It is not difficult to see that the canonical map $\dR_r(X) \to \dV_r(X)$ is injective, Proposition \ref{bas.charpoly}.
One might ask, what   the image of $\dR_r(X)$ in  $\dV_r(X)$ is.

With the notation as above we can also define the set $\dV_r(X,D)$ of 2-skeleton sheaves with bounded wild
ramification, see Definition~\ref{def.boundram}.
Deligne expresses the hope that the following question about existence of $\ell$-adic sheaves  might have a positive answer.

\begin{question}\label{fund.quest}
Is the map
$\dR_r(X,D) \to \dV_r(X,D)$ bijective
for any Cartier divisor $D\in \Div^+(\bar X)$ with support in $\bar X \setminus X$?
\end{question}

To motivate the question one should think of the set of curves $\Cu(X)$ together with the systems of intersections
of curves as the $2$-skeleton of $X$.
To be more precise, the analogy is as follows:
For a $CW$-complex $S$ let $S_{\le d}$ be the union of $i$-cells of $S$ ($i\le d$), i.e.\ its
$d$-skeleton. Assume that $S_{\le 0}$ consists of just one point.

\bigskip

\begin{center}
\renewcommand{\arraystretch}{1.4}
\begin{tabular}[c]{|p{5cm}|p{5cm}|}
\hline  $CW$-complex $S$ (with $S_{\le 0}=\{ *\}  $) &  Variety $X/\F_q$  \\
\hline $1$-sphere $S^1$ with topological fundamental group $\pi_1(S^1)=\Z$  & Finite field
$\F_q$ with Weil group $W(\F_q) = \Z$ \\
\hline  $S^1$-bouquet $S_{\le 1} $ &  Set of closed points $|X|$  \\
\hline $2$-cell in $S$ &  Curve in $\Cu(X)$ \\
\hline Relation in $\pi_1(S)$ coming from $2$-cell & Reciprocity law on curve \\
\hline $2$-skeleton $S_{\le 2}$ & System of curves $\Cu(X)$\\
\hline Local system on $S$ & Lisse $\Ql$-Weil sheaf on $X$\\
\hline
\end{tabular}
\end{center}

\bigskip

In the sense of this analogy, Deligne's Question~\ref{fund.quest} is the analog of the
fact that the fundamental groups of $S$ and $S_{\le 2}$ 
 are the same  \cite[Thm.\ 4.23]{Hat}, except that 
 we consider only the information contained in $\ell$-adic representations, in addition only modulo semi-simplification,  and that
 there
is no analog of wild ramification over $CW$-complexes.


\medskip

For $D=0$ a positive answer to Deligne's question is given by Drinfeld \cite[Thm~2.5]{Drinfeld}. His proof uses a method developed by  
  Wiesend \cite{W} to reduce the problem to Lafforgue's theorem.
 For
$r=1$ and $D=0$ it was first shown by Schmidt--Spiess \cite{SchSp} using motivic cohomology, and later by
Wiesend \cite{W2} using more
elementary methods.

\medskip

The stronger form  of Deligne's finiteness theorem says that Theorem~\ref{thm.finite1} remains true for 2-skeleton sheaves. We
say that a {\it 2-skeleton sheaf } $V\in \dV_r(X)$ {\it on a connected scheme} $X$ is {\it irreducible}  if it cannot be written in the
from $V_1 \oplus V_2$ with $V_i\in \dV_{r_i}(X)$ and  $r_1,r_2>0$.
In   Appendix~\ref{app.B},  Proposition~\ref{propcurveex},  we give a proof of the well known fact that a sheaf $V\in \dR_r(X)$ is irreducible if and only if its image in
$\dV_r(X)$ is irreducible.

 The main result of this note now says:

\begin{thm}[Deligne] \label{thm.finite2}
 Let $X\in \Sm_{\F_q}$ be  connected and $D\in \Div^+(\bar X)$ be an effective  Cartier divisor 
  supported in $\bar X \setminus X$.  The 
 set of irreducible 2-skeleton sheaves $V\in \dV_r(X,D)$ is finite up to twist by elements from $\dR_1(\F_q)$. Its cardinality does not depend on $\ell\neq p$.
\end{thm}

The theorem implies in particular that for a given integer $N>0$ there are only finitely many
$V\in \dV_r(X,D)$  with $\det(V)^{\otimes N}=1$.
Following Deligne we will reduce the theorem to the one-dimensional case, where it is a well
known consequence of the Langlands correspondence of Drinfeld--Lafforgue. Some hints how
the one-dimensional case is related to the theory of automorphic forms are given in Section~\ref{sec.prdim1}. 
The proof of Theorem~\ref{thm.finite2} is completed in Section~\ref{sec.irrcomp}.

\medskip

\noindent
{\it Idea of proof.}
The central idea of Deligne is to define an algebraic moduli space structure on the set
$\dV_r(X,D)$, such that it becomes an affine scheme of finite type over $\Q$. In fact
$\dV_r(X,D)$ will be the $\bar \Q_\ell$-points of this moduli space. One shows
that the irreducible components of the moduli space over $\bar \Q_\ell$ are  `generated' by
certain
twists of 2-skeleton sheaves, which implies the finiteness theorem, because there are only
finitely many irreducible components.

Firstly, one constructs the moduli space structure of finite type over $\Q$ for $\dim(X)=1$. Then one immediately
gets an algebraic structure on $\dV_r(X,D)$ in the higher dimensional case and the central
point is to show that  $\dV_r(X,D)$ is of finite type over $\Q$ for higher dimensional $X$ too.

The main method to show the finite type property  is a result of Deligne (Theorem~\ref{frob.thm}), relying on Weil II and the Langlands correspondence, which says that for one-dimensional $X$ there is a natural number $N$
depending logarithmically on the genus of $\bar X$ and the degree of $D$ such that $V\in
\dV_r(X,D)$ is determined by the polynomials $f_V(x)$ with $\deg(x)\le N$. 
Here for $V\in \dV_r(X,D)$ we denote by $f_V(x)$ the characteristic polynomial of the Frobenius
at the closed point $x\in |X|$, see Section~\ref{sec.basics} for a precise definition.

\smallskip




\section{Ramification theory}

\noindent
In this section  we review some 
facts from ramification theory. We work over the finite field $\F_q$. In fact all results
remain true over a perfect base field of positive characteristic and for lisse \'etale $\ell$-adic sheaves.


\subsection{Local ramification} 
We follow \cite[Sec. 2.2]{Laumon}.
Let $K$ be a complete discretely valued field  with perfect residue field of
characteristic $p>0$. Let $G=\Gal(\bar K / K)$, where $\bar K$ is a separable closure of $K$. There is a descending filtration
$(I^{(\lambda)})_{0\le \lambda \in \R}$ by closed normal subgroups of $G$ with the following
properties:
\begin{itemize}
\item $\bigcap_{\lambda' < \lambda} I^{(\lambda')} = I^{(\lambda)}$,
\item $ \bigcap_{\lambda \in \R} I^{(\lambda)}=0 $,
\item $I^{(0+)}$ is the unique maximal pro-$p$ subgroup of the inertia group $I^{(0)}$, where $I^{(\lambda +)}$ is defined as $\overline{
    \bigcup_{\lambda' > \lambda} I^{(\lambda')}  }$.
\end{itemize}

Let $G \to \GL(V )$ be a continuous representation on a finite dimensional $\Ql$-vector
space $V$ with $\ell\ne p$.

\begin{defn}
The {\em Swan conductor} of $V$ is defined as 
\[
\Sw(V)= \sum_{\lambda >0} \lambda \dim(V^{I^{(\lambda +)}} / V^{I^{(\lambda )}}) . 
\]
\end{defn}

The Swan conductor is additive with respect to extensions of  $\ell$-adic Galois representations,  it does not
change if we replace $V$ by its semi-simplification.



For later reference we recall the behavior of the Swan conductor with respect to direct
sum and tensor product.
If $V,V'$ are two $\Ql$-$G$-modules as above and $V^\vee$ denotes the dual representation,  then
\begin{align}\label{Swprop1}
\Sw(V\oplus V' ) &=  \Sw(V) + \Sw(V')\\
\frac{\Sw(V\otimes V')}{\rank(V) \rank(V')} &\le \frac{\Sw(V)}{\rank(V)} + \frac{\Sw(V')}{\rank(V')}    \\
\label{Swprop3}  \Sw(V^\vee ) &= \Sw(V)
\end{align} 

\subsection{Global ramification $(\dim=1)$} \label{sec.gloramcur}

Let $X/\F_q$ be a smooth
connected curve with
smooth compactification $X\subset \bar X$. Let $V$ be in $\dR_r(X)$. 

The {\it Swan conductor} $\Sw(V)$ is defined to be  the effective Cartier divisor $$\sum_{x\in | \bar X|} \Sw_x(V) \cdot
[x] \in \Div^+(\bar X) .$$ 
Here $\Sw_x(V)$ is the Swan conductor of the restriction of the representation class $V$ to
the complete local field ${\rm frac}( \widehat{\sO_{\bar X,x}} )$.
We say that $V$ is {\em tame} if $\Sw(V)=0$.

Clearly the Swan conductor of $V$ is the same as the Swan conductor of any twist $\chi\cdot V,  \ \chi\in \dR_1(\F_q)$.

Let $\phi:X'\to X$ be an \'etale covering of smooth curves with compactification $\bar
\phi : \bar X' \to \bar X$.
 By $\Dis_{\bar X'/\bar X}\in \Div^+(\bar X)$ we denote the discriminant
\cite{Serre} of $\bar X'$ over $\bar X$, cf.\ Section~\ref{sec.gloram}.

\begin{lem}[Conductor-discriminant-formula]\label{conddisc} 
For $V\in \dR_r(X)$ with
 $\phi^*(V)$ tame
the inequality of divisors
\[
\Sw(V)  \le   \rank (V)\, \Dis_{\bar X'/\bar X}
\]
holds on $\bar X$.
\end{lem}

\begin{proof}
By abuse of notation we write $V$ also for a sheaf representing $V$.
There is an injective map of sheaves on $X$
\[
V \to \phi_* \circ\phi^*(V)
\]
For any $x\in |X|$ 
\[
\Sw_x(V)\le \Sw_x(\phi_* \circ\phi^*(V))  \le \rank(V)\,
{\rm mult}_x (  \Dis_{\bar X'/\bar X} ) .
\]
The second inequality follows from \cite[Prop.\ 1(c)]{Raynaud}. 
\end{proof}

\bigskip

\begin{defn}
Let $D\in \Div^+(\bar X)$ be an effective Cartier divisor. 
The subset  $\dR_r(X,D)\subset \dR_r(X)$ is defined by the condition
$\Sw(V)\le D$. If 
$V\in \dR(X)$ lies in $ \dR_r(X,D)$, we say that its  {\it ramification is bounded by $D$}.

\end{defn}

Let $\F_{q^n}$ be the algebraic closure of $\F_q$ in $k(X)$.

\begin{defn}\label{def.complex}
For a divisor $D\in \Div^+(\bar X)$ 
we define the {\it complexity} of $D$ to be
\[
\sC_{D}= 2 g(\bar X) + 2 \deg_{\F_{q^n}}(D) + 1,
\]
where $g(\bar X)$ is the genus of $\bar X$ over $\F_{q^n}$ and $\deg_{\F_{q^n}}$ is the
degree over $\F_{q^n}$. 
\end{defn}

\begin{prop}\label{dimestimate}
Assume $X/\F_q$ is geometrically connected.
For $D\in \Div^+(\bar X)$  with
$\mathrm{supp}(D)=\bar X\setminus X$ and  for $V\in \dR_r(X,r D)$, the inequality
\[
\dim_{\bar{\Q}_\ell} H^0_c(X\otimes_{\F_q}\F,V) + \dim_{\bar{\Q}_\ell} H^1_c(X \otimes_{\F_q} \F,V) \le
{\rm rank} (V)\,  \sC_D 
\]
holds.
\end{prop}

\begin{proof}
Grothendieck-Ogg-Shafarevich theorem says that
\[
\chi_c(X \otimes_{\F_q} \F,V) = (2-2g(\bar X))\, {\rm rank}(V) -
\sum_{x\in \bar X\setminus X} 
( \rank(V) +s_x(V)) ,
\]
see \cite[Th\'eor\`eme 2.2.1.2]{Laumon}.
Furthermore 
\ml{}{{\rm dim} \ H^0_c(X \otimes_{\F_q} \F, V)\le r \ {\rm and} \\ 
{\rm dim}  \ H^2_c(X \otimes_{\F_q} \F,V)={\rm dim} \ H^0(X\otimes_{\F_q} \F,V^\vee) \le r. \notag
}
\end{proof}

\subsection{Global ramification $ (\dim \ge 1)$}\label{sec.gloram}

We follow an idea of Alexander Schmidt for the definition of the discriminant for higher
dimensional schemes. 

  Let $X$ be a connected scheme in $ \Sm_{\F_q}$.
Let $X \subset \bar X$ be a normal compactification of $X$ over $k$ such that $\bar X \setminus X$
is the support of an effective Cartier divisor on $\bar X$. Clearly, such a
compactification always exists.

Let $\Cu(X)$ be the set of normalizations of closed integral subschemes of $X$ of
dimension one. For $C$ in $\Cu(X)$ denote by $\phi: C\to X$ the natural morphism. By $\bar
C$ we denote the smooth compactification of $C$ over $\F_q$ and by $\bar \phi: \bar C\to \bar
X$ we denote the canonical extension.

 Recall that in Section~\ref{sec.intro} we introduced the set of lisse $\Ql$-Weil sheaves
$\dR_r(X)$ and of 2-skeleton sheaves $\dV_r(X)$ on $X$ of rank $r$.

\begin{defn}\label{def.boundram}
For $V\in \dR_r(X)$ or $V\in \dV_r(X)$ and $D\in \Div^+(\bar X)$ an effective Cartier with support in $\bar X
\setminus X$ we (formally) write $\Sw(V)\le D$ and say that the {\it ramification of $V$ is bounded by $D$ } if for every curve $C\subset \Cu(X)$ we have
\[
\Sw ( \phi^* (V))  \le \bar \phi^*(D)
\]
 in the sense of Section~\ref{sec.gloramcur}.
The subsets  $\dR_r(X,D)\subset \dR_r(X)$  and $\sV_r(X, D)\subset \dV_r(X)$ are  defined by the condition
$\Sw(V)\le D$.
\end{defn}

In the rest of this section we show that for any $V\in \dR_r(X)$ there is an effective
divisor $D$  with $\Sw(V)\le D$.

\medskip

Let $\psi: X' \to X$ be an \'etale covering (thus  finite) and let $\bar \psi : \bar X' \to \bar
X$ be the finite, normal extension of $X'$ over $\bar X$.

\begin{defn}[A.\ Schmidt]\label{def.disc}
The  {\em
discriminant} $\sI(\Dis_{\bar X'/\bar X})$ is the coherent ideal sheaf in $\sO_{\bar X}$
locally generated by all elements
\[ 
 \det ({\rm Tr}_{K'/K} ( x_i \, x_j  ) )_{i,j}   
\]
where $x_1, \ldots ,x_n\in \psi_* (\sO_{\bar X'})$ are local sections restricting to a basis of
$K'$ over $K$. Here $K= k(X)$ and $K'=k(X')$.
\end{defn}

Clearly, $\sI(\Dis_{\bar X'/\bar X})|_{X} = \sO_X$. This definition extends the classical
definition for curves \cite{Serre},  in which case $\sI(\Dis_{\bar X'/\bar X})=\sO_{\bar X}(-\Dis_{\bar X'/\bar X})$,
where $X\subset \bar X$  and $X'\subset \bar X'$ are the smooth compactifications.

The following lemma is easy to show.

\begin{lem}[Semi-continuity]\label{lem.semcont}
In the situation of Definition~\ref{def.disc}
let $\bar \phi: \bar C \to \bar X$ be a smooth curve mapping to $\bar X$ with $C=\bar
\phi^{-1}(X)$ non-empty. Let $C'$ be a connected component of $ C \times_X X'$ and let $C'\hookrightarrow \bar C'$ be the smooth
compactification.   Then 
\[
  \bar \phi^{-1} (\sI(\Dis_{\bar X'/\bar X})) \subset \sO_{\bar C}(-\Dis_{\bar C'/\bar C})  .
\]
\end{lem}

\bigskip

\begin{prop}\label{prop.rambound}
For $V\in \dR_r(X)$ there is an effective Cartier divisor $D\in \Div^+(\bar X)$ such that $\Sw(V) \le D$.
\end{prop}

\begin{proof}
By Remark~\ref{weil.etale} we can assume that $V$ is an \'etale sheaf on $X$. Then there is a
 local field $ E\subset \Ql$ finite over $\Q_\ell$ with ring of integers $\sO_E$ such that $V
$ comes form an $\ell$-adic $\sO_E$-sheaf $V_1$. Let $\hat E$ be the finite residue field of $\sO_E$.
There is  a connected \'etale covering $\psi:X'\to X$ such that $\psi^*(V_1
\otimes_{\sO_E} \hat E)$ is trivial. This implies that $\psi^*(V)$ is tame. Let $D_1\in
\Div^+(\bar X)$ be an effective Cartier divisor with support in $\bar X \setminus X$ such that $\sO_{\bar
  X}(-D_1) \subset \sI(\Dis_{\bar X'/\bar X})$ and set $D=\rank(V) D_1$. With the notation of
Lemma~\ref{lem.semcont} we obtain
\[
 \bar \phi^*(D_1) \ge \Dis_{\bar C'/\bar C}
\]
 As the pullback of
$V$ to $C'$ is tame we obtain from Lemma~\ref{conddisc} the first inequality in
\[
\Sw( \phi^*(V) ) \le \rank (V) \Dis_{\bar C'/\bar C} \le  \bar \phi^*(D).
\]
\end{proof}

\begin{rmk}
We do not know any example for a $V\in \dV_r(X)$ for which there does  {\em not} exist  a
divisor $D$ with $\Sw(V)\le D$.  If such an example existed, it would in particular show,
in view of Proposition~\ref{prop.rambound}, that not all 2-skeleton sheaves are actual
sheaves.
\end{rmk}

We conclude this section by a remark on the relation of our ramification theory with the theory of
Abbes-Saito \cite{AS}.
 We expect that 
for $V\in \dR_r(X)$,  $\Sw(V) \le D$ is
equivalent to the following: For every  open immersion $X\subset  X_1$ over $\F_q$ with
the property that
$X_1 \setminus X$ is a simple normal crossing divisor and for any morphism
$h: X_1 \to \bar X$,  the Abbes-Saito log-ramification Swan conductor of $ h^*(V)$ at a maximal point
of $ X_1 \setminus X$ is $\le$  the multiplicity of $h^*(D)$ at the maximal point. 

For $D=0$ this equivalence is shown in \cite{KeSch} relying on \cite{W}. For $r=1$ it is
known modulo resolution of singularities
by work of I.\ Barrientos (forthcoming Ph.D. thesis, Universit\"at Regensburg).

\section{$\ell$-adic sheaves}

\subsection{Basics}\label{sec.basics}

For $X\in \Sm_{\F_q}$ we defined in Section~\ref{sec.intro} the set $\dR_r(X)$ of lisse $\bar
\Q_\ell$-Weil sheaves on $X$ of rank $r$   up to isomorphism and up to semi-simplification and the set $\dV_r(X)$ of 2-skeleton sheaves.
Clearly, $\dR_r$ and $\dV_r$ form contravariant functors from $\Sm_{\F_q}$ to the category
of sets.

For $V\in \dR_r(X)$,  taking the characteristic polynomials of Frobenius defines a function
\[
f_V: |X|  \to  \bar{\Q}_{\ell}[t], \ \ \ f_V(x)=\det(1-t\, F_x, V_{\bar x}).
\]
For $V\in \dV_r(X)$ we can still define $f_V(x)$ by choosing a curve $C\in \Cu(X)$ such
that $C\to X$ is a closed immersion in a neighborhood of $x$ and we set $f_V(x)=f_{V_C}(x)$. It follows from
the definition that $f_V(x)$ does not depend on the choice of $C$. 

We define the trace 
\[
t^n_V: X(\F_{q^n})  \to  \bar{\Q}_{\ell}, \ \ \ t^n_V(x)={\rm tr}( F_x, V_{\bar x})
\]
for $V\in \dR_r(X)$ and similarly for $V\in \dV_r(X)$.

\medskip

We define $\sP_r$ to be the affine scheme over $\Q$ whose points $\sP_r(A)$ with values in
a
$\Q$-algebra $A$ consist of  the set of polynomials $1+a_1 t + \cdots + a_r t^r \in A[t]$ with
$a_r\in A^\times$. Mapping $(\alpha_i)_{1\le i\le r}$ with $\alpha_i\in A^\times$  to $$(1-\alpha_1t) \cdots
(1-\alpha_r t)\in A[t]$$ defines a scheme isomorphism 
\begin{equation}\label{eq.Gmiso}
\mathbb G_m^r / S_r
\xrightarrow{\simeq} \sP_r,
\end{equation}
 where $S_r$ is the permutation group of $r$ elements.

For $d\ge 1$  the finite morphism $\G_m^r\to \G_m^r$ which sends $(\alpha_1,\ldots, \alpha_r)$ to $(\alpha_1^d,\ldots, \alpha_r^d)$
descends to $\sP_r$ 
to define the 
finite scheme homomorphism $\psi_d:\sP_r \to \sP_r$.

Let $\sL_r(X)$ be the product $\prod_{|X|} \sP_r$ with one copy of $\sP_r$ for every
closed point of $X$. It is an affine scheme  over $\Q$ which if ${\rm dim}(X)\ge 1$  is not of finite type over
$\Q$.
Denote by $\pi_x:\sL_r(X) \to \sP_r$ the projection onto the factor corresponding to $x\in |X|$.
 We make $\sL_r$ into a contravariant functor from $\Sm_{\F_q}$ to the category of
affine schemes over $\Q$ as follows: Let $f:Y\to X$ be a morphism of schemes in
$\Sm_{\F_q}$. The image of $(P_x)_{x\in  |X|}\in \sL_r(X)$ under pullback by $f$ is defined to be
\[
\left(\psi_{[k(y):k(f(y))]} P_{f(y)}  \right)_{y\in |Y|} \in \sL_r(Y) .
\]

For $N>0$ we  similarly define $\sL^{\le N}_r(X)$ to be the product over all $x\in |X|$
with $\deg(x)\le N$ over $\F_q$, with the corresponding forgetful morphism $\sL_r(X)\to \sL_r^{\le N}(X)$.

\medskip

Putting things together we get morphisms of contravariant functors
\begin{equation} \label{fun.inj}
\dR_r(X)   \longrightarrow \dV_r(X)   \xrightarrow{\kappa : V\mapsto f_V} \sL_r(X)(\bar \Q_\ell). 
\end{equation}

\begin{prop}\label{bas.charpoly}
For $X\in \Sm_{\F_r}$ the maps $\dR_r(X)   \to \sL_r(X)(\bar \Q_\ell)$ and $\dV_r(X)   \xrightarrow{\kappa}
\sL_r(X)(\bar \Q_\ell)$ are injective.
\end{prop}

\begin{proof}
We only have to show the injectivity for $\dR_r(X)$, since the curve case for $\dR_r(X)$
implies already the general case for $\dV_r(X)$.  We can easily recover the trace
functions $t^n_V$ from the characteristic polynomials $f_V$. The Chebotarev density
theorem \cite[Ch.~6]{FrJar}
implies that the traces of Frobenius determine semi-simple sheaves, see \cite[Thm.~1.1.2]{Laumon}.
\end{proof}

In Section~\ref{sec.frob}  we will prove a much stronger result, saying that a finite number of characteristic
polynomials $f_V(x)$ are sufficient to recover $V$ up to twist, as long as $V$ runs over $\ell$-adic
sheaves with some fixed bounded ramification and fixed rank.

For later reference we recall the relation between Weil sheaves and \'etale sheaves from
Weil II \cite[Prop.~1.3.4]{WeilII}.   We say that $V\in \dR_r(X)$ is {\it \'etale}  if it comes from a  lisse \'etale
$\Ql$-sheaf on $X$.

\begin{prop}\label{weil.etale.curve}
For $X$ connected and $V\in \dR_1(X)$, which we consider as a continuous homomorphism
$V:W(X)\to \Ql^\times$, the geometric monodromy group  ${\rm im}(\pi_1(X_{\bar \F}))\subset  W(X)/\ker(V)$ is finite, in
particular the monodromy group  $ W(X)/\ker(V)$ is discrete.
The sheaf $V$ extends to a continuous homomorphism $ \pi_1(X) \to \Ql^\times$, i.e.\  V is
\'etale, if and only if $\im
(V) \subset \bar \Z_\ell^\times $.
\end{prop}

\begin{prop}\label{weil.etale.gen}
For  $X$ connected an irreducible $V\in \dR_r(X)$ is \'etale
 if and only if its determinant $\det(V)$ is \'etale. In
 particular there is always a twist  $\chi \cdot V$ with $\chi\in \dR_1(\F_q)$ which is \'etale.
\end{prop}

\medskip

\subsection{Implications of Langlands}

In this section we recall some consequences of the  Langlands correspondence of Drinfeld
and Lafforgue \cite{Laf} for the theory of
$\ell$-adic sheaves.



The following theorem is shown in \cite[Th\'eor\`eme VII.6]{Laf}.

\begin{thm}\label{Lafforgue}
For $X\in \Sm_{\F_q}$  connected of dimension one and for $V\in \dR_r(X)$ irreducible with
determinant of
finite order the following holds: 
\begin{itemize}
\item[(i)] For an arbitrary, not necessarily continuous, automorphism
$\sigma\in {\rm Aut}(\bar \Q_\ell/\Q)$, there is a $V_\sigma \in \dR_r(X)$, called $\sigma$-companion, such that $$f_{V_\sigma} = \sigma (f_V), $$ where $\sigma$ acts
on the polynomial ring
$ \bar \Q_\ell[t]$ by $\sigma$ on $\bar \Q_\ell$ and 
 by $\sigma(t) = t$.
\item[(ii)] $V$ is pure of weight $0$.
\end{itemize}
\end{thm}

\medskip

Later,  we deduce from the theorem that $\sigma$-companions exist for
arbitrary $V\in \dR_r(X)$ in dimension one, not necessarily of finite determinant, see Corollary~\ref{corocompan}.

For  $\dim(X)$ arbitrary and $V\in \dR_1(X)$, which we consider as a continuous homomorphism $V:W(X)\to \Ql^\times$, the
$\sigma$-companion $V_\sigma$ simply corresponds to the continuous map $\sigma \circ V: W(X)
\to \Ql^\times$. In fact $\sigma \circ V$ is continuous, because $W(X)/ \ker(V)$ is
discrete by Proposition~\ref{weil.etale.curve}.


From Lafforgue's theorem one can deduce certain results on higher dimensional schemes.

\begin{cor}  \label{corodetweight}
Let $X$ be a connected scheme in $\Sm_{\F_q}$ of arbitrary dimension.
For an irreducible $V\in \dR_r(X)$ the
following are
equivalent:
\begin{itemize}
\item[(i)] $V$ is pure of weight $0$,
\item[(ii)] there is a closed point $x\in X$ such that $V_{\bar x}$ is pure of
weight $0$,
\item[(iii)] there is  $\chi\in \dR_1(\F_q)$ pure of weight $0$ such that the determinant 
$\det( \chi \cdot V )$ is of finite order.
\end{itemize}
\end{cor}

\begin{proof}
(iii) $\Rightarrow$ (i):\\
For a closed point $x\in X$
choose a curve $C/k$ and a morphism $\phi:C\to X$ such that $x$ is in the set
theoretic image of $\phi$ and such that $\phi^* V$ is irreducible. 
A proof of the existence of such a curve is given in an appendix, 
Proposition \ref{propcurveex}.
Then by Theorem~\ref{Lafforgue} the sheaf $\phi^* V$ is pure of weight $0$ on $C$, so
$V_{\bar x}$ is also pure of weight $0$.

\smallskip

 (i) $\Rightarrow$ (ii): Trivially. 

\smallskip

(ii) $\Rightarrow$ (iii):\\
Choose $\chi\in \dR_1(\F_q)$ such that
$(\chi|_{k(x)})^{\otimes r} = \det (V_{\bar x})^\vee$. 
By Proposition~\ref{weil.etale.curve}
it follows that the determinant
$\det(\chi \cdot V ) $ has finite order.
\end{proof}

Let $\mathcal W$ be the quotient of $\bar \Q_\ell^\times$ modulo the numbers of weight $0$ in
the sense of \cite[Def.~1.2.1]{WeilII} (algebraic numbers all complex conjugates of which have absolute value $1$).

\begin{cor}\label{corwdec}
A sheaf $V\in \dR_r(X)$, resp.   a 2-skeleton  sheaf $V\in \dV_r(X)$, can be decomposed uniquely as a sum
\[
V= \bigoplus_{w\in\mathcal W} V_w
\]
with the property that  $V_w\in \dR(X)$, resp.\ $V_w\in \dV(X)$,  such that for each point $x\in |X|$, all eigenvalues of the Frobenius $F_x$ on $V_w$
lie in the class $w$.
\end{cor}




\medskip

\begin{cor}  \label{corocompan}
Assume $\dim(X)=1$.
For  $V\in \dR_r(X)$ and
an automorphism $\sigma\in {\rm
Aut}(\bar \Q_\ell/\Q) $, there is a $\sigma$-companion to $V$,
i.e.\ $V_\sigma \in \dR_r(X)$  such that $$f_{V_\sigma} =
\sigma (f_V)
.$$
\end{cor}

\begin{proof}
Without loss of generality  we may assume that  $V$ is irreducible. 
In the same way as in the proof of Corollary~\ref{corodetweight} we find
$\chi\in \dR_1(\F_q)$
 such that
$\chi \cdot V$ has determinant of finite order.
A $\sigma$-companion of $\chi \cdot V$ exists by Theorem~\ref{Lafforgue} and a
$\sigma$-companion of $\chi$ exists by the remarks below Theorem~\ref{Lafforgue}.
As the formation of $\sigma$-companions is compatible with tensor products, $V_\sigma =
(V\cdot \chi)_\sigma \cdot (\chi_\sigma)^\vee $ is a $\sigma$-companion of $V$.
\end{proof}

Deligne showed a compatibility result \cite[Thm.~9.8]{Delepsilon}  for the Swan conductor of $\sigma$-companions.

\begin{prop}\label{ramcompatible}
Let $V$ and $V_\sigma$ be $\sigma$-companions on a one-dimensional $X\in \Sm_{\F_q}$ as in as in Corollary~\ref{corocompan}. Then $\Sw(V)= \Sw(V_\sigma)$.
\end{prop}

\medskip

Recall from \eqref{fun.inj} that there is a canonical injective map  of sets  $\dV_r(X)\xrightarrow{\kappa}
\sL_r(X)(\Ql) $. In the following corollary we use the notation of Section~\ref{sec.gloram}.

\begin{cor}\label{vir.compan}
For $X\in \Sm_{\F_q}$ and an effective Cartier divisor $D\in \Div^+(\bar X)$ with support
in $\bar X \setminus X$  the action of ${\rm Aut}(\Ql /\Q)$ on $\sL_r(X)(\Ql)$ stabilizes
$\alpha(\dV_r(X))$ and $\alpha(\dV_r(X,D))$.
\end{cor}


\begin{rmk}
Drinfeld has shown \cite{Drinfeld} that Corollary~\ref{corocompan} remains true for higher
dimensional $X\in \Sm_{\F_q}$. His argument
relies on Deligne's Theorem~\ref{mainthmirr}.
\end{rmk}

\smallskip







\subsection{Proof of Thm.\ \ref{thm.finite1} $(\dim =1)$}\label{sec.prdim1}

Theorem~\ref{thm.finite1} for one-dimensional schemes is a well-known consequence of
Lafforgue's Langlands correspondence  for ${\rm GL}_r$ \cite{Laf}. Let $X\in \Sm_{\F_q}$ be
of dimension one with smooth compactification $\bar X$, $L=k(X)$.  The Langlands
correspondence says that there is a natural bijective equivalence between cuspidal
automorphic irreducible representations $\pi$ of ${\rm GL}_r(\mathbb A_L)$ (with values in
$\Ql$)  and
continuous irreducible representations of the Weil group $\sigma_\pi:W(L) \to {\GL}_r(\Ql)$,
which are unramified almost everywhere. For such an automorphic $\pi$ one defines an (Artin)
conductor ${\rm Ar}(\pi)\in \Div^+(X)$ and one constructs an open compact subgroup $K\subset
{\rm GL}_r(\mathbb A_L)$
depending only on ${\rm Ar}(\pi)$ such that the space of $K$ invariant vectors of $\pi$ has
dimension one, see \cite{JPS}. 

The divisor ${\rm Ar}(\pi)$ has support in $\bar X \setminus X$ if and only if
 $\sigma_\pi$ is unramified over $X$. Moreover $$\Sw_x(\sigma_\pi) +r \ge {\rm Ar}_x
 (\pi)  $$ for $x\in |\bar X|$. 

For an arbitrary compact open subgroup $K \subset
{\rm GL}_r(\mathbb A_L)$ the number of cuspidal automorphic irreducible representations
$\pi$ with fixed central character and which have a non-trivial $K$-invariant vector is
finite by work of Harder, Gelfand and Piatetski-Shapiro, see
\cite[Thm.~9.2.14]{LaumonDrinfeld}.

Via the Langlands correspondence this implies that for given $D\in \Div^+(\bar X)$  with
support in $\bar X \setminus X$ and for given $W\in \dR_1(X)$  the number of irreducible $V\in
\dR_r(X)$ with $\det(V)=W$ and  with $\Sw(V)\le D$ is finite. Recall that the determinant
of $\sigma_\pi$  corresponds to the central character of $\pi$ via class field theory. 

\medskip
As written by Deligne in an email to us dated July 30, 2012, 
one can also use the quasi-orthogonality relations from Claim  \ref{orth}  and  a sphere packing argument to conclude, but Claim \ref{orth}  relies on purity, which again comes from the Langlands correspondence on curves, so we do not make the argument explicit.

\subsection{Structure of a lisse $\bar{\Q}_\ell$-sheaf over a scheme
over a finite field. }

\noindent 
Let the notation be as above.
The following proposition is shown in \cite[Prop.\ 5.3.9]{BBD}.

\begin{prop}\label{structure}
Let $V$ be irreducible  in $\dR_r(X)$. 
\begin{itemize}
\item[(i)] Let $m$ be the number of irreducible constituents of $V_{\F}$. There
is a unique irreducible  $V^\flat\in \dR_{r/m}(X_{\F_{q^m}})$ such that
\begin{itemize}
\item the pullback of $V^\flat$ to $X\otimes_{\F_q} \F$ is irreducible,
\item $V= b_{m,*} V^\flat $, where $b_{m}$ is the natural map $ X\otimes_{\F_q}
\F_{q^{m}}\to X$.
\end{itemize}
\item[(ii)] $V$ is pure of weight $0$ if and only if  $V^\flat$ is pure of
weight $0$.
\item[(iii)] If $V'\in \dR_r(X)$ is another sheaf
on
$X$ with $V'_{\F} = V_\F$, then there is a unique  sheaf $W\in \dR_1(\F_{q^m})$
  with $$V' = b_{m,*}(V^\flat\otimes W) .$$
\end{itemize}
\end{prop}

\medskip

A special case of the Grothendieck trace formula \cite[(1.1.1.3)]{Laumon} says:

\begin{prop} \label{tr}
Let $V$ and $m$ be as in Proposition~\ref{structure}. For $n\ge 1$
and $x\in X(\F_{q^n} )$  
$$t^n_V (x) = \sum_{y\in {X_{\F_{q^m}}(\F_{q^n})} \atop y \mapsto x} 
t^n_{V^\flat}(y).$$
\end{prop}

Concretely, 
$t^n_V (x) = 0$
if $m$ does not divide $n$.

\section{Frobenius on curves}\label{sec.frob}







\noindent
We now present Deligne's key technical method for proving his finiteness theorems. It
strengthens Proposition~\ref{bas.charpoly} on curves by allowing us to recover an $\ell$-adic sheaf
from an effectively determined finite number of characteristic polynomials of Frobenius.

Our notation is explained in Section~\ref{sec.intro} and Section~\ref{sec.basics}.
Throughout this section  $X$ is a geometrically connected scheme in $\Sm_{\F_q}$ with $\dim(X)=1$.

\begin{thm}[Deligne] \label{frob.thm}
The natural map   
\[
\dR_r(X,D) \xrightarrow{\kappa_N} \sL^{\le N}_r(X)(\bar \Q_\ell)
\]
is injective if 
\begin{equation}
N\ge  4 r^2\lceil \log_{q}(2r^2\sC_{ D} ) \rceil
\end{equation}
\end{thm}

Here for a real number $w$ we let $\lceil w \rceil$ be the smallest integer
larger or equal to $w$.
Theorem~\ref{frob.thm} relies on the Langlands correspondence and weight arguments form
Weil II.  The Langlands correspondence enters via  Corollary~\ref{corwdec}. 

 We
deduce Theorem~\ref{frob.thm} from the following trace version:
\begin{prop} \label{thmtrace}
If $V,V'\in \dR_r(X,D)$  are pure of weight $0$ and
satisfy
$t^n_V  = t^n_{V'}$ for all 
\begin{equation}\label{tracenb}
n \le  4 r^2  \lceil
\log_q (2 r^2\, \sC_{D} ) \rceil,
\end{equation}
then $V=V'$.
\end{prop}

\begin{proof}[Prop.\ \ref{thmtrace} $\Rightarrow$ Thm.\ \ref{frob.thm}]
Let $V, V' \in \dR_r(X,D)$. 
We write
\[
V=\bigoplus_{w\in \mathcal W} V_w \quad\text{ and }\quad  V'=\bigoplus_{w\in \mathcal W} V'_w
\]
as in Corollary~\ref{corwdec}. 
The condition  $\alpha_N(V)=\alpha_N(V')$  implies $\alpha_N(V_w)=\alpha_N(V'_w)$, thus
 $t^n_{V_w}=t^n_{V'_w}$ for all $w\in \sW $ and all
$n$ as in \eqref{tracenb}. By Proposition~\ref{thmtrace}, applied to some twist of weight $0$ of $V_w$ and $V'_w$ by the same $\chi$, this implies $V_w= V'_w$ for all
$w\in \mathcal W$. 
\end{proof}

\subsection{Proof of Proposition \ref{thmtrace}}

Let $J$ be the set of irreducible $W\in \dR_s(X)$, $1\le s\le r$, which are twists of direct summands of
$V\oplus V'$. Set $I=J/\text{twist}$.
Choose representative sheaves $S_i\in \dR(X)$ which are pure of weight $0$  ($i\in I$).
In particular this implies that $\Hom_{X\otimes_{\F_q}\F}(S_{i_1}, S_{i_2} )=0$ for $i_1\ne i_2\in I$
by Proposition~\ref{structure}. Also for each $i\in I$ we have
\[
S_i = b_{m_i, *} S^\flat_i 
\]
for positive integers $m_i$ and geometrically irreducible $S^\flat_i\in \dR(X_{\F_{q^{m_i}}} )$ with the
notation of Proposition~\ref{structure}.

It follows from Proposition~\ref{structure} that there are 
$W_i,W'_i\in \dR(\F_{q^{m_i}})$ pure of weight $0$
 such that
\[
V =  \bigoplus_{i\in I} b_{m_i, * } (S^\flat_i \otimes_{\bar{\Q}_\ell } W_i)
\]
and
\[
V' =  \bigoplus_{i\in I} b_{m_i, * } (S^\flat_i \otimes_{\bar{\Q}_\ell } W'_i).
\]

For $n>0$ set \[
I_n = \{ i\in I, \: m_i | n  \} .
\]

\begin{lem}\label{lemlinind} Let $S_i$ be in $\sR(X)$  pairwise distinct, geometrically irreducible, pure of weight $0$.  
Then the  functions $$t^n_{S_i}: X(\F_{q^n}) \to \bar{\Q}_\ell \;\;\;\;\;\;\;\;\;\;
(i\in
I_n) $$  are linearly independent over $\Ql$ for $n\ge 2\, \log_q(2 r^2 \sC_{D}) $.
\end{lem}

\begin{proof}
Fix an isomorphism $\iota:\bar{\Q}_\ell \stackrel{\sim}{\to} \C$. Assume we have
a
linear relation 
\begin{equation}\label{linreltrace}
\sum_{i\in I_n} \lambda_i \, t^n_{S_i} = 0, \ \lambda_i\in \bar \Q_\ell,
\end{equation}
 such that not all $\lambda_i$ are $0$. Multiplying by a constant in $\bar
\Q_\ell^\times$, we may assume that 
 $|\iota(\lambda_{i_\circ})| = 1$ for one $i_\circ\in I_n$ and
$|\iota(\lambda_{i})|\le 1$ for all $i\in I_n$. 
Set
\[
\langle S_{i_1}, S_{i_2} \rangle_n =  \sum_{x\in X(\F_{q^n})}
t^n_{\mathit{Hom}(S_{i_1}, S_{i_2})}(x) \]
for $i_1, i_2\in I_n$.
Observe that
\[
t^n_{\mathit{Hom}(S_{i_1}, S_{i_2})}= t^n_{S_{i_1}^\vee}\cdot t^n_{S_{i_2}}.
\]
Multiplying \eqref{linreltrace} by $t^n_{S_{i_\circ}^\vee}$ and summing over all
$x\in X(\F_{q^n})$  one obtains
\begin{equation}
\sum_{i\in I_n} \lambda_i\, \langle S_{i_\circ} , S_i \rangle_n = 0 .
\end{equation}

\begin{claim} \label{orth}
One has
\begin{itemize}
\item [(i)]
\[
| \iota  \langle S_{i_\circ} , S_i \rangle_n  | \le {\rm rank} (S_{i_\circ} )
{\rm rank} (
S_i)\, \sC_{D}\, q^{n/2}
\]
for $i\ne i_\circ$,
\item[(ii)] 
\[
| m_{i_\circ}\, q^n  - \iota    \langle S_{i_\circ} , S_{i_\circ} \rangle_n   
|\le   {\rm rank} (S_{i_\circ} )^2 \, \sC_{D}\, q^{n/2} .
\]
\end{itemize}
\end{claim}

\medskip

\noindent {\em Proof of} (i):\\
By \cite[Th\'eor\`eme 3.3.1]{WeilII}  the eigenvalues $\alpha$ of
$F^n$ on $H^k_c(X\otimes_{\F_q} \F, \mathit{Hom}(S_{i_\circ}, S_{i_\circ}))$ for $k\le 1$ fulfill
$$
|\iota \alpha | \le q^{n/2} .$$
On the other hand 
\begin{eqnarray*}
\dim_{\bar{\Q}_\ell} (H^0_c(X\otimes_{\F_q} \F, \mathit{Hom}(S_{i_\circ}, S_{i}))  ) +
\dim_{\bar{\Q}_\ell} (H^1_c(X\otimes_{\F_q} \F, \mathit{Hom}(S_{i_\circ}, S_{i}))  ) & \le &
\\
 {\rm rank} (S_{i_\circ}) {\rm rank} ( S_{i})\, \sC_{D}   & &
\end{eqnarray*}
by Proposition~\ref{dimestimate}. In fact the we have 
$$\Sw( \mathit{Hom}(S_{i_\circ},
S_{i})) \le   {\rm rank} (S_{i_\circ}) {\rm rank} ( S_{i}) D  $$
by \eqref{Swprop1} - \eqref{Swprop3}.
Under the assumption $i\ne i_\circ$ one has
\[
H^2_c(X\otimes_{\F_q} \F, \mathit{Hom}(S_{i_\circ}, S_{i})) =  {\rm Hom}_{X\otimes_{\F_q} \F} (  S_i ,
S_{i_\circ}) \otimes \bar \Q_\ell (-1) =0
\]
by Poincar\'e duality. Putting this together and using Grothendieck's
trace formula \cite[1.1.1.3]{Laumon}  one obtains (i).

\medskip

\noindent{\em Proof of} (ii):\\
It is similar to (i) but this
time we have
\[
\dim_{\bar \Q_\ell} H^2_c(X\otimes_{\F_q} 
F, \mathit{Hom}(S_{i_\circ}, S_{i})) = m_{i_\circ}
\]
and for an eigenvalue $\alpha$ of $F^n$ on $$H^2_c(X \otimes_{\F_q} \F,
\mathit{Hom}(S_{i_\circ}, S_{i}))=  {\rm Hom}_{X\otimes_{\F_q} 
\F} (  S_i , S_{i_\circ})
\otimes \bar \Q_\ell (-1)$$
we have $\alpha= q^n$. This finishes the proof of the claim.

\bigskip

Since under the assumption on $n$ from Lemma~\ref{lemlinind}
\[
\sC_D  \, {\rm rank} (S_{i_\circ} )\sum_{i\in I_n} {\rm rank}( S_i)
< q^{n/2},
\]
we get a contradiction to the linear dependence \eqref{linreltrace}.

\end{proof}


By Proposition~\ref{tr} for any $n\ge 0$ we have 
\[
t^n_V = \sum_{i\in I_n }t^n_{W_i}\, t^n_{  S_i  }
\]
and 
\[
t^n_{V'} = \sum_{i\in I_n } t^n_{W'_i}\, t^n_{  S_i  }.
\]
Under the assumption of equality of traces from Theorem~\ref{thmtrace} and using
Lemma~\ref{lemlinind} we get 
\begin{equation}\label{eqtracef}
{\rm Tr}(F^n, W_i) = {\rm Tr}(F^n, W'_i)    \;\;\;\;\;\;\;\;\;\; i\in I_n
\end{equation}
for $$  2\, \log_q(2r^2 \sC_D) \le n \le 4 r^2  \lceil
\log_q (2 r^2\, \sC_{D} ) \rceil .$$
In particular this means that equality \eqref{eqtracef} holds for $$n\in \{ m_i
\, A, m_i\,(A+1),\ldots , m_i\, (A+2r-1) \}, $$
where $A= \lceil 2\, \log_q(2 r^2 \sC_D) \rceil $. So Lemma~\ref{lemlinalg}
 applied to the set $\{b_1,\ldots, b_w\}$ of eigenvalues 
of $F^{m_i}$ of  $W_i$ and $ W'_i$ (so $w\le 2r$)
shows that $ W_i = W'_i $ for all $i\in I$.

\begin{lem}\label{lemlinalg}
Let $k$ be a field and consider elements $a_1,\ldots , a_w \in k, \ b_1, \cdots
b_w\in k^\times$ such that
\[
F(n) := \sum_{1\le j\le w} a_j \, b_j^{n} = 0
\]
for $1\le n \le w$. Then $F(n)=0$
for all $n \in \Z$.
\end{lem}

\begin{proof}
Without loss of generality we can assume that the $b_j$ are pairwise different
for $1\le j\le w$. Then the
Vandermonde matrix
\[
(b_j^n)_{1\le j,n\le w} 
\]
has non-vanishing determinant, which implies that $a_j=0$ for all $ j$.
\end{proof}

\section{Moduli space of $\ell$-adic sheaves}\label{sec.moduli}

\noindent
In Section~\ref{sec.basics} we introduced an injective map 
\[
\kappa: \dV_r(X) \to \sL_r(X)(\Ql)
\] 
from the set of 2-skeleton $\ell$-adic sheaves to the $\Ql$-points of an affine scheme
$\sL_r(X)$ defined over $\Q$, which is not of finite type over $\Q$ if ${\rm dim}(X)\ge 1$. 
Assume that there is a connected normal projective compactification $X \subset \bar X$ such that $\bar X
\setminus X$ is the support of an effective Cartier divisor on $\bar X$. We use the notation
of Section~\ref{sec.basics}.

The existence of the moduli space of $\ell$-adic sheaves on $X$ is shown in the following theorem
of Deligne.
\begin{thm}\label{thm.moduli}
For any effective Cartier divisor $D\in \Div^+(\bar X)$ with support in $\bar X \setminus
X$ there is a unique reduced closed subscheme $L_r(X,D)$ of $\sL_r(X)$ which is of finite type
over $\Q$ and such that $$L_r(X,D)(\Ql) = \kappa (\dV_r(X,D)).$$
\end{thm}

Uniqueness is immediate from Proposition~\ref{HNS}.
In Section~\ref{sec.modcu} we construct $L_r(X,D)$ for $\dim(X)=1$. In
Section~\ref{sec.modhigh} we construct $L_r(X,D)$ for general $X$. Before we begin the
proof we introduce some elementary constructions on $\sL_r(X)$.

\subsection{Direct sum and twist as scheme morphisms}


For $r=r_1+r_2$ the isomorphism
\[
\G_m^{r_1} \times_\Q \G_m^{r_2}  \xrightarrow{\simeq} \G_m^r
\]
together with the  embedding of groups $S_{r_1}\times S_{r_2}\subset S_{r_1+r_2}$ induces a finite surjective map
\ga{+}{  - \oplus - :   \sP_{r_1}\times_{\Q} \sP_{r_2}\to \sP_{r}, \quad  (P,Q)\mapsto P Q}
via the isomorphism \eqref{eq.Gmiso}. We call it the direct sum.

\medskip
There is a twisting action by $\G_m$
\begin{equation*}
\G_m \times_\Q \sP_r  \to \sP_r  , \quad (\alpha , P) \mapsto \alpha \cdot P
\end{equation*}
  defined by the diagonal action of $\G_m$
on $\G_m^r$ 
\[
(\alpha , (\alpha_1 ,\ldots , \alpha_r)) \mapsto (\alpha \cdot \alpha_1 ,\ldots , \alpha
\cdot \alpha_r) 
\]
and the isomorphism  \eqref{eq.Gmiso}.

\medskip

We now extend the direct sum and twist morphisms to $\sL(X)$.

By taking direct sum on any factor of $\sL(X)$ we get for $r_1+r_2=r$ a morphism of
schemes over $\Q$
\begin{equation}\label{acsum}
- \oplus - : \sL_{r_1}(X) \times \sL_{r_2}(X) \to \sL_r(X)
\end{equation}
Note that the direct sum is not a finite morphism in general, since we have an
infinite product over closed points of $X$.

\medskip

The twist is an action of $\G_m$
\begin{equation}\label{actwist}
\G_m \times_\Q \sL_r(X) \to \sL_r(X)
\end{equation}
given by
\[
(\alpha, (P_x)_{x\in |X|} ) \mapsto \alpha \cdot (P_x)_{x\in |X|} = (\alpha^{\deg(x)} \cdot P_x
)_{x\in |X|} 
\]
where we take the degree of a point $x$  over $\F_q$.

\medskip

Let $k$ be a field containing $\Q$ and $P_i\in \sL_{r_i}(k) $, $i=1, \ldots , n$.
Assume $r_i>0$ for all $i$ and set
$r=r_1 + \cdots + r_n$.

\begin{lem}\label{lem.finitemor} 
The morphism of schemes over the field $k$
\[
\rho: \G_m^{n}  \to \sL_r(X) ,\quad (\alpha_i)_{i=1,\ldots ,n} \mapsto \alpha_1 \cdot P_1
\oplus \cdots \oplus  \alpha_n \cdot P_n
\]
is finite.
\end{lem}

\begin{proof}
In fact already the composition of $\rho$ with the projection to one factor $\sP_r$ of
$\sL_r(X)$, corresponding to a point $x\in |X|$, is finite. To see this write this morphism
as the composition of finite morphisms over $k$
\[
 \G_m^{n}   \xrightarrow{\psi_{\deg(x)}} \G_m^n  \xrightarrow{\cdot (P_1,\ldots , P_n)}
 \sP_{r_1}\times \cdots \times \sP_{r_n}  \xrightarrow{\oplus} \sP_r  .
\]
\end{proof}






\subsection{Moduli over curves}\label{sec.modcu}

In this section we prove Theorem~\ref{thm.moduli} for $\dim(X)=1$.
The dimension one case of Theorem~\ref{thm.finite1} was shown in Section~\ref{sec.prdim1}.
In particular we get:

\begin{lem} \label{lem.finitedec}
 There are up to
twist only {\em finitely many} irreducible direct summands of the sheaves $V\in
\dR_r(X,D)=\dV_r(X,D)$. 
\end{lem}

\smallskip

{\em Step 1:}\\
Consider $V_1 \oplus \cdots \oplus V_n \in \dR_r(X,D)$ and the map
\begin{equation}\label{modcurmap1}
(\dR_1(\F_q))^n \to \sL_r(X)(\Ql) , \quad (\chi_1 , \ldots ,\chi_n) \mapsto \kappa(  \chi_1\cdot
V_1 \oplus \cdots \oplus \chi_n\cdot .
V_n )
\end{equation}
 This map is just the induced map on $\Ql$-points of the finite scheme morphism over
$k=\Ql$ from Lemma~\ref{lem.finitemor}, where we take $P_i=\kappa(V_i)$.
By Proposition~\ref{points} there is a unique reduced closed subscheme $L(V_i)$   of
$\sL_r(X)\otimes \Ql$ of finite type over $\Ql$
such that $L(V_i)(\Ql)$ is the image of the map \eqref{modcurmap1}.

\smallskip

{\em Step 2:}\\
By Lemma~\ref{lem.finitedec} there are only finitely many  direct sums
\begin{equation}\label{eqsumdec}
V_1 \oplus
\cdots \oplus V_n \in \dR_r(X,D)
\end{equation}
 with $V_i$ irreducible  up to  twists  $\chi_i \mapsto  \chi_i \cdot V_i$ with $\chi_i \in \dR_1(\F_q)$.
Let $$L_r(X,D)_\Ql\hookrightarrow \sL_r(X)\otimes_\Q \Ql$$ be the reduced scheme, which is
the union of the finitely many closed subschemes $L(V_i)\hookrightarrow
\sL_r(X)\otimes_\Q \Ql$ corresponding to representatives  of the finitely many twisting
classes of direct sums \eqref{eqsumdec}.
Clearly $ L_r(X,D)_\Ql(\Ql) = \kappa(\dR_r(X,D))$ and $ L_r(X,D)_\Ql$ is of finite type over $\Ql$.

\smallskip

{\em Step 3:}\\
By Corollary~\ref{vir.compan} the automorphism group ${\rm Aut}(\Ql/\Q)$ acting on
$\sL_r(X)$ stabilizes $\kappa(\sR_r(X,D))$. Therefore
by the descent Proposition~\ref{descent} the scheme $L_r(X,D)_\Ql\hookrightarrow \sL_r(X)\otimes_\Q
\Ql$ over $\Ql$ descends to a closed
subscheme $L_r(X,D)\hookrightarrow \sL_r(X)$. This is the {\em moduli space of $\ell$-adic
sheaves} on curves, the existence of which was
claimed in Theorem~\ref{thm.moduli}.

\medskip

From the proof of Lemma~\ref{lem.finitemor} and the above construction we deduce:

\begin{lem}\label{lem.finmorpro}
For any $x\in |X|$ the composite map $$L_r(X,D) \to \sL_r(X) \xrightarrow{\pi_x}   \sP_r$$ is a finite
morphism of schemes.
\end{lem}

\subsection{Higher dimension}\label{sec.modhigh}

Now the dimension $d=\dim(X)$ of $X\in \Sm_{\F_q}$ is allowed to be arbitrary. In order to prove Theorem~\ref{thm.moduli} in
general we first construct a closed subscheme $L_r(X,D) \hookrightarrow \sL_r(X)$ such
that $$L_r(X,D)(\Ql)=\kappa(\dV_r(X,D))$$ relying on Theorem~\ref{thm.moduli} for curves.
However from this construction it is not clear that $L_r(X,D)$ is of finite type over $\Q$.
The main step is to show that it is of finite type using Theorem~\ref{frob.thm}.

\smallskip

{\em Step 1:}\\
We define the reduced closed subscheme $L_r(X,D)\hookrightarrow \sL_r(X) $ by the Cartesian
square (in the category of reduced schemes)
\[
\xymatrix{
L_r(X,D)  \ar[r] \ar[d] &   \sL_r(X)  \ar[d] \\
\prod\limits_{C\in \Cu(X)} L_r(C,\bar \phi^*(D)) \ar[r]   & \prod\limits_{C\in \Cu(X)} \sL_r(C)
}
\]
where $\Cu(X)$ is defined in Section~\ref{ss.strong}. 
Clearly, from the curve case of Theorem~\ref{thm.moduli} and the definition of $\dV_r(X,D)$ we get $$L_r(X,D)(\Ql)=\kappa(\dV_r(X,D)).$$
 In addition, as $\sL(X)\to  \prod\limits_{C\in \Cu(X)} \sL_r(C)$ is a closed immersion, so is $\sL(X,D)\to \prod\limits_{C\in \Cu(X)} L_r(C,\bar \phi^*(D))$.

\smallskip

{\em Step 2:}\\
Let $C$ be a purely one-dimensional scheme which is separated and of finite type over
$\F_q$. Let $\phi_i:E_i \to C$ ($i=1,\ldots ,m$) be the normalizations of the irreducible
components of $C$  and let $\phi: E = \coprod_i E_i \to C$ be the disjoint union. Let $D\in \Div^+(\bar E)$
be an effective divisor  with supports in
$\bar E \setminus E$. Here $\bar E$ is the canonical smooth compactification of
$E$. 
Define the reduced scheme $L_r(C,D)$  by the Cartesian square (in the category of reduced schemes)
\[
\xymatrix{
L_r(C,D)  \ar[r] \ar[d] &   \prod_{j=1 , \ldots , m} L_r(E_j,D_j)  \ar[d] \\
  \prod_{i=1 , \ldots , m} L_r(E_i,D_i) \ar[r] &   \prod_{i\ne j}
  L_r((E_i\times_C E_j)_{\rm red} ) 
}
\]

\smallskip

{\em Step 3:}\\
By an {\em exhaustive system of curves} on $X$ we mean a sequence $(C_n)_{n\ge 0}$ of 
purely one-dimensional closed subschemes $C_n\hookrightarrow X$ with the
properties (a) -- (d) listed below. We write $\phi: E_n \to X$ for the normalization of $C_n$. For a divisor $D'\in
\Div^+(\bar E_n)$ we let $\sC_{D'}$ be the maximum of the complexities of the irreducible
components of $E_n\otimes \F$, see Definition~\ref{def.complex}. 
\begin{itemize}
\item[(a)]  $C_n\hookrightarrow C_{n+1}$ for $n\ge 0$,
\item[(b)] $E_n(\F_{q^n})  \to X(\F_{q^n})$ is surjective,
\item[(c)] the fields of constants of the irreducible components of $E_n$ ($n \ge 0$) are
  bounded,
\item[(d)] the complexity $\sC_{\bar \phi_n^*(D)}$ of $E_n$ satisfies
\[
\sC_{\bar \phi_n^*(D)} = O(n).
\]
\end{itemize}

\begin{lem}\label{lem.excurve}
 Any
 $X\in \Sm_{\F_q}$ admits an exhaustive system of curves.
\end{lem}

The proof of the lemma is given below.

 \medskip

Let now $(C_n)$ be an exhaustive system of curves on $X$. Set $D_n=\bar \phi_n^*(D)\in \Div^+(\bar E_n)$. An immediate consequence of (a)--(d) and the Riemann hypothesis for curves is that for $n\gg 0$
any irreducible component of $C_{n+1}$ meets $C_n$. This implies by Lemma~\ref{lem.finmorpro} that the
tower of affine schemes of finite type over $\Q$
\[
\cdots \to L_r(C_{n+1},D_{n+1}) \xrightarrow{\tau}  L_r(C_{n},D_{n}) \to \cdots
\]
has finite transition morphisms.
 Clearly, $L_r(X,D)$ maps to this tower. Since the complexities of the irreducible curves
 grow linearly in $n$ and the fields of constants are bounded,  Theorem~\ref{frob.thm} implies that there is
 $N\ge 0$ such that the map
\[
 L_r(C_{n+1},D_{n+1})(\Ql) \to \sL_r^{\le n}(E_{n+1})
\]
is injective for $n\ge N$. As this map factors through 
\[
\tau:  L_r(C_{n+1},D_{n+1})(\Ql) \to  L_r(C_{n},D_{n}) (\Ql)
\]
by (b), we get injectivity of $\tau $ on $\Ql$-points for $n\ge N$.
 Consider the
intersection of the images $$I_n= \bigcap_{i\ge 0} \tau^{i}(L_r(C_{n+i}, D_{n+i})
)\hookrightarrow   L_r(C_{n},D_{n}) , $$
endowed with the reduced closed subscheme structure. Then the transition maps in the tower
\[
\cdots \to I_{n+1} \to I_n \to \cdots
\]
are finite and induce bijections on $\Ql$-points for $n\ge N$. By Proposition~\ref{proinj} we get an $N'\ge
0$ such that $I_{n+1}\to I_n$ is an isomorphism of schemes for $n\ge N'$. 
The closed immersion $L_r(X,D)\to  \sL_r(X)$ factors through
$L_r(X,D)\to \varprojlim_n  L_r(C_{n},D_{n})$, which is therefore itself a closed immersion. Thus 
 we obtain  a
closed immersion
\[
L_r(X,D) \to \varprojlim_n  L_r(C_{n},D_{n}) \cong \varprojlim_n I_n \xrightarrow{\simeq} I_{N'},
\]
and therefore $L_r(X,D)$ is of finite type over $\Q$.

\begin{proof}[Proof of Lemma~\ref{lem.excurve}]
Using Noether normalization we find a finite number of finite surjective morphisms 
\[
\bar \eta_s : \bar X \to \P^d, \quad s=1,\ldots  , w
\]
with the property that every point $x\in |X|$ is in the \'etale locus of one of the $\eta_s=
\bar \eta_s |_{X}$. See \cite[Theorem~1]{Kedlaya} for more details.

\begin{claim}
For a point $y\in \P^d(\F_{q^n})$ there is a morphism $\phi_y: \P^1\to
\P^d$ of degree $< n$ with $y\in \phi_y(\P^1(\F_{q^n}))$.
\end{claim}

\begin{proof}[Proof of Claim]
The closed point $y$ lies in an affine chart $$\A^d_{\F_q} =\Spec(\F_q [T_1,\ldots, T_d] )\hookrightarrow \P^d_{\F_q}$$
and gives rise
to a
homomorphism $\F_q [T_1,\ldots, T_d] \to \F_{q^n}$. We choose an embedding
$ \Spec \F_{q^n} \hookrightarrow \A^1_{\F_q}=\Spec(\F_q[T])$ and a lifting
$$\F_q [T_1,\ldots, T_d]\to \F_q[T] $$ with $\deg(\phi(T_i))< n$ ($1\le
i\le d$).
By projective completion we obtain a morphism $\phi_y: \P^1_{\F_q} \to \P^d_{\F_q}$
of degree less than $n$ factoring the morphism  $y\to \P^d$.

\end{proof}

For $x\in |X|$ of degree $n$ choose a lift $x\in X(\F_{q^n})$ and an $s$ such that $x$ is
in the \'etale locus of $\eta_s$. Furthermore choose $\phi_y:\P^1 \to \P^d$ as in the claim with
$y=\eta_s(x)$. Clearly $x$ lifts to a smooth point of
$(\P^1\times_{\P^d} X) (\F_{q^n}) $ contained in an irreducible component which we call
$Z$. Let $\phi_x:C_x\to X$ be the normalization of the image of $Z$ in $X$. Then $x\in
\phi_x(C_x(\F_{q^n}))$.

\medskip

We assume now that we have made the choice of the curve $\phi_x:C_x\to X$ above for any
point $x\in |X|$. As usual $\bar \phi_x:\bar C_x\to \bar X$ denotes the smooth
compactification of $C_x$.
From the Riemann-Hurwitz  formula \cite[Cor.~2  ]{Har} we deduce the growth property
\[
\sC_{\bar \phi_x^*(D)} = \sO( \deg(x) )
\]
for the complexity of $\bar C_x$. Furthermore it is clear that the fields of constants of
the curves $C_x$ are bounded. Therefore the subschemes
\[
C_n = \bigcup_{\deg(x)\le n} \phi_x(C_x)  \hookrightarrow X
\]
satisfy the conditions (a)--(d) above.
\end{proof}

\section{Irreducible components and proof of finiteness theorems}\label{sec.irrcomp}

Recall that we defined irreducible 2-skeleton  sheaves  in Section~\ref{sec.intro} and that in Section~\ref{sec.moduli} we constructed an affine scheme 
$L_r(X,D)$ of finite type over $\Q$, the $\bar{\Q}_\ell$-points of which are in bijection
with 2-skeleton sheaves of rank $r$ with ramification bounded by $D$.
For this we had to assume that $\bar X$ is a normal projective variety defined over $\F_q$ and $D$ is an effective Cartier divisor supported in $\bar X\setminus X$.

The following theorem describes the irreducible components of $L_r(X,D)$ over $\bar \Q$ or, what
is the same, over $\Ql$. 

\begin{thm} \label{irrcpt}
\begin{itemize}
\item[1)] Given $V_1,\ldots, V_m$ irreducible in $\dV(X)$ such that $V_1\oplus \ldots \oplus V_m\in \sV_r(X,D)$, there is a unique irreducible component $Z\hookrightarrow L_r(X,D)\otimes \bar{\Q}$ such that 
\ga{cpt}{   Z(\bar{\Q}_\ell)=\{\kappa(\chi_1\cdot V_1\oplus \ldots \oplus \chi_m\cdot
  V_m)\, |\, \chi_i\in \dR_1(\F_q)\} }
\item[2)] If $Z\hookrightarrow L_r(X,D)\otimes \bar{\Q}$ is an irreducible component, then there are $V_1,\ldots, V_m$ irreducible in $\dV(X, D)$  such that \eqref{cpt} holds true.

\end{itemize}
\end{thm}
\begin{proof}
We first prove 2). 
Let $d$ be the dimension of $Z$, so $\bar{ \Q}(Z)$ has transcendence degree $d$
over $\bar{\Q}$. Let   
 $\kappa(V)\in Z(\bar{\Q}_\ell)$ be a geometric generic point, corresponding to
$\iota: \bar{\Q}(Z)\hookrightarrow \bar{\Q}_\ell$.

By definition, the coefficients of the local polynomials $f_V(x), \ x\in |X|$
span $\iota(\bar{\Q}(Z))$.  The subfield $K$ of $\bar{\Q}_\ell$ spanned by the
(inverse) roots of the $f_V(x)$ is algebraic over  $\iota(\bar{\Q}(Z))$, and
thus has transcendence degree $d$
over $\bar{\Q}$ as well.

 Writing
\ga{eq.decweight}{V=\oplus_{w\in \sW} V_w}
thanks to Corollary~\ref{corwdec},  the number
$m $ of 
 such $w$ with $V_w\neq 0$ is $\ge d $. Indeed those $w$ have the property that
they span $K$.

On the other hand,
the map \eqref{modcurmap1} corresponding to the decomposition \eqref{eq.decweight}
 is the $\bar{ \Q}_\ell$-points of a finite map with source $\G_m^m$, which
is irreducible, and has image contained in $Z$. So we conclude $ m =d $ and
that the morphism $\G_m^m\to Z$ is finite surjective.

We prove 1). By Corollary~\ref{corwdec}, the $V_i$ have  the property that there
is a $w_i\in \sW$ such that all the inverse eigenvalues of the  Frobenius
$F_x$ on $V_i$ lie in the class of $w_i$. Replacing $V_i$ by $\chi_i\cdot V_i$
for adequately chosen  $\chi_i \in \sR_1(\F_q)$, we may assume that  $w_i\neq
w_j$ in $\sW$ if $i\neq j$.  We consider the irreducible reduced closed
subscheme $Z\hookrightarrow L_r(X,D)\otimes \Ql$ defined by its $\Ql$-points $
\{\kappa(\chi_1\cdot V_1\oplus \ldots \oplus \chi_m\cdot V_m) \ |  \ \chi_i\in
\dR_1(\F_q)     \}$.   Let $ Z'$ be an irreducible component of $L_r(X,D)\otimes \Ql$  containing $Z$. 
Thus by B), $$Z'(\bar{\Q}_\ell)=\{\kappa(\chi'_1\cdot V_1'\oplus \ldots
\chi'_{m'}\cdot V'_{m'}) \ | \ \chi'_i\in \sR_1(\F_q)\}.$$
So there are $\chi'_i$ such that 
\ga{decomp}{ V_1\oplus \ldots \oplus V_m= \chi'_1V'_1\oplus \ldots \oplus
\chi'_{m'}V'_{m'}    }
  As $V'_j$ is irreducible for any $j \in \{1, \ldots, m'\}$, it is
of class $w$ for some $w\in \sW$ in the sense of Corollary~\ref{corwdec}.
So for each $j\in \{1,\ldots, m'\}$, there is a $i\in \{1,\ldots, m\}$ with
 $\chi'_j\cdot V'_j\subset V_i$, and thus $\chi'_j\cdot V'_j = V_i$
as $V_i$ is irreducible. This implies $m=m'$ and the decompositions
\eqref{decomp}  are the same, up to ordering. 
So $Z=Z'$.
\end{proof}

\begin{cor} \label{cor.irr}
A 2-skeleton sheaf $V\in \sV_r(X,D)$ is irreducible if and only if $\kappa(V)$ lies on a
one-dimensional 
irreducible component of $L_r(X,D)\otimes \Ql$. 
In this case $\kappa(V)$ lies on a unique irreducible component $Z/\Ql$. The component $Z$ has the form
$$Z(\Ql) =  \{ \kappa( \chi\cdot V )\, |\,  \chi\in \dR_1(\F_q) \}$$
 and it does not meet any other irreducible component.
\end{cor}

\begin{rmk}
If Question~\ref{fund.quest} had a positive answer and using a more refined analysis of
Deligne~\cite{DelFinitude} one could deduce that the moduli space $L_r(X,D)$ is smooth and any
irreducible component is of the from $\G_m^{s_1}\times \A^{s_2}$ ($s_1,s_2\ge 0$).
\end{rmk}

\begin{proof}[Proof of Theorem~\ref{thm.finite2}]
Using the Chow lemma \cite[Sec.~5.6]{EGAII} we can assume without loss of generality that $\bar X$ is projective.
  By Corollary~\ref{cor.irr}, the set of one-dimensional irreducible components of
  $L_r(X,D)\otimes \bar{\Q}$ is in bijection with the set of irreducible 2-skeleton sheaves
  on $X$
  up to twist by $\dR_1(\F_q)$. Since $L_r(X,D)$ is of finite type, there are only
  finitely many irreducible components.
\end{proof}

\begin{proof}[Proof of Theorem~\ref{mainthmirr}]

By Corollary~\ref{vir.compan}  there is a natural action of ${\rm Aut}(\Ql /\Q)$ on
$\dV_r(X,D)$ compatible via $f_V$ with the action on $\Ql[t]$ which fixes $t$. Let $N>0$
be an integer such that $\det(V)^{\otimes N}=1 $.
For $\sigma \in {\rm Aut}(\Ql /\Q)$ we then have 
\[
1 = \sigma (\det(V)^{\otimes N} ) =  \det( \sigma (V))^{\otimes N} .
\]
Then Theorem~\ref{thm.finite2}, (see also the remark following the  theorem), implies that the orbit of $V$ under ${\rm Aut}(\Ql /\Q)$ is
finite. Let $H\subset {\rm Aut}(\Ql /\Q)$ be the stabilizer group of $V$. As $[{\rm
  Aut}(\Ql /\Q) :H]< \infty$  we get that $E(V)=\Ql^{H}$
is a number field.
\end{proof}

\noindent
In order to effectively determine the field $E(V)$ for $V\in \dR_r(X)$ with $X\in
\Sm_{\F_q}$ projective   one can use the following simple consequence of a theorem of Drinfeld \cite{Drinfeld}, which itself
relies on Deligne's Theorem~\ref{mainthmirr}.

\begin{prop} \label{hyperplane}
For $X/\F_q$ a smooth projective geometrically connected scheme and $H\hookrightarrow X$ a
smooth hypersurface section with $\dim(H)>0$ consider $V\in \dR_r(X)$. Then $E(V) = E(V|_H )$.
\end{prop}

\begin{proof}
Observe that the Weil group of $H$ surjects onto the Weil group of $X$, so we get an
injection $\dR_r(X)\to \dR_r(H)$. By
\cite{Drinfeld} Corollary~\ref{corocompan} remains true for higher dimensional
smooth schemes $X/\F_q$, i.e.\ for any $\sigma\in {\rm Aut}(\bar \Q_{\ell} / \Q
)$ there exists a $\sigma$-companion $V_\sigma$ to $V$. By the above injectivity, the sheaves $V$ and
$V|_H$ have the same stabilizer $G$ in ${\rm Aut}(\Ql /\Q)$. We get $$E(V)= \Ql^G = E(V|_H).  $$
\end{proof}

\section{Applications}\label{appl}

We now explain applications of the finiteness theorem for $2$-skeleton sheaves to a conjecture from Weil II \cite[Conj.~1.2.10~(ii)]{WeilII} and to
Chow groups of $0$-cycles.

\subsection{Finiteness of relative Chow group of $0$-cycles}\label{sec.finchow}

It was shown by Colliot-Th\'el\`ene--Sansuc--Soul\'e   \cite{CTSS} and by 
Kato--Saito \cite{KS}  that over a finite field, 
the Chow group of $0$-cycles  of degree $0$ of a proper variety is
finite.

Assume now that $X\subset \bar X$ is a compactification as above and let $D\in \Div^+(\bar X)$
be an effective Cartier divisor with support in $\bar X \setminus X $. For a curve $
C\in \Cu(X)$ and an effective  divisor $E\in \Div^+(\bar C)$ with support in $\bar C \setminus C$, where $\bar C$ is the smooth compactification of $C$,   let 
\[
P_{k(C)}(E) = \{ g\in  k(C)^\times |  {\rm ord}_x (1-g)\ge {\rm mult}_x( E)+1  \text{ for }  x\in \bar C \setminus C \} 
\]
be the unit group with modulus  well known from the ideal theoretic version of global class
field theory. Set
\[
\CH_0(X,D)= Z_0(X) / \im [ \bigoplus_{C\in \Cu(X)} P_{k(C)}( \bar \phi^* D)  ].
\] 
Here $\bar \phi : \bar C \to \bar X$ is the extension of the natural morphism $\phi:C \to X$.
A similar Chow group of $0$-cycles is used in  \cite{ESV}, \cite{Russell} to define 2-skeleton Albanese
varieties. For $D=0$ and $\bar X \setminus X$ a simple normal crossing divisor it is
isomorphic to the Suslin homology group $H_0(X)$ \cite{Schm}. For $\dim(X)=1$ it is the classical ideal class
group with modulus $D+ E$, where $E$ is the reduced divisor with support $ \bar X \setminus X$. 

From Deligne's finiteness Theorem~\ref{thm.finite2} and class field theory one immediately
obtains a finiteness result which was expected to hold in higher dimensional class field theory.

\begin{thm} \label{CH}
For any $D\in \Div^+(\bar X)$ as above the kernel of the degree map from  $\CH_0(X,D)$ to $\Z$ is finite. 
\end{thm}

\smallskip

\subsection{Coefficients of characteristic polynomial of the Frobenii at closed points}\label{sec.coeff}

In  \cite[Conjecture~1.2.10]{WeilII} Deligne conjectured that sheaves $V\in \dR_r(X)$ with
certain obviously necessary properties should behave as if they all came from geometry,
i.e.\  as if they were $\ell$-adic realizations of pure motives over $X$. In particular they should not only be `defined over'
$\Ql$, but over $\bar \Q$. In this section we explain how this latter conjecture  of
Deligne (for the
precise formulation see Corollary~\ref{delconjweil} below), follows from
Theorem~\ref{thm.finite2}.

In fact Corollary~\ref{delconjweil} is the main result of Deligne's article \cite{DelNumberField}.
His  proof uses Bombieri's upper estimates for the $\ell$-adic Euler characteristic of an affine variety defined over
a finite field, (and Katz' improvement for the Betti numbers)  in terms of the embedding dimension, the number and the degree of the defining equations,  which rests, aside of Weil II, 
on Dwork's $p$-adic methods.
  In \cite{EK} it was
observed that one could replace the use of $p$-adic cohomology theory by some
more elementary ramification theory. After this  Deligne extended his methods in  \cite{DelFinitude}
to obtain the Finiteness
Theorem \ref{thm.finite2}.

For $V\in \dV_r(X)$ and $x\in | X|$ one defines the characteristic polynomial  of Frobenius
$f_V(x)\in \Ql [t]$
at the point $x$, see Section~\ref{sec.basics}. Let $E(V)$ be the subfield of $\Ql$
generated by all coefficients of all the polynomials $f_V(x)$ where $x\in |X|$ runs through the closed points.

\begin{thm}\label{mainthmirr}
Let $D\in \Div^+(\bar
X)$ be an effective Cartier divisor  with support in $\bar X \setminus X$.
For $V\in \dV_r(X,D)$ irreducible with $\det (V)$ of finite order,  the field
$E(V)$ is a number field.
\end{thm}

In Section~\ref{sec.irrcomp} we deduce Theorem~\ref{mainthmirr} from Theorem~\ref{thm.finite2}.
By associating to $V\in \dR_r(X)$ its 2-skeleton sheaf in $\dV_r(X)$, one
 finally obtains Deligne's conjecture  \cite[Conj.~1. 2.10(ii) ]{DelNumberField} from Weil II. 

\begin{cor}\label{delconjweil}
For $V\in \dR_r(X)$ irreducible with $\det (V)$ of finite order the field
$E(V)$ is a number field.
\end{cor}

In fact by
 Proposition~\ref{prop.rambound} there is a divisor $D$ such that $V\in \dR_r(X,D)$. Then
 apply Theorem~\ref{mainthmirr} to the induced 2-skeleton sheaf in $\dV_r(X,D)$.

\section{ Deligne's conjecture on the number of irreducible lisse sheaves of rank $r$ over a smooth curve with prescribed local monodromy at infinity} \label{conjecture}
Let $C$ be a smooth quasi-projective geometrically irreducible curve over $\F_q$, and $C\hookrightarrow \bar C$ be a smooth compactification. 
One fixes an algebraic closure $\F\supset \F_q$  of $\F_q$. 
For each point $ s\in (\bar{C}\setminus C)(\F)$, one fixes a  $\bar \Q_\ell$-representation $V_{ s}$ of the inertia 
$$I( s)={\rm Gal}(K^{\rm sep}_{ s}/ K_{ s} )$$
 where $K_s$ is the completion of the function field $K=k(C)$ at $s$.  We write 
  $$I(\bar s)=P\rtimes \prod_{\ell'\neq p} \Z_\ell(1),$$ 
  where $P$ is the wild inertia, a pro-$p$-group. A  generator $\xi_{\ell'}$ of  $\Z_{\ell'}(1), \ \ell'\neq p$, acts on $V_{ s}$ for all $ s\in  (\bar{C}\setminus C)(\F)$ .
  Since the open immersion $j: C\hookrightarrow \bar{ C}$ is defined over $\F_q$, if $ s\in  (\bar{C}\setminus C)(\F)$ is defined over $\F_{q^n}$,  for any conjugate point $ s'  \in  (\bar{C}\setminus C)(\F)$, the group $I( s')$ is conjugate to $I( s)$ by ${\rm Gal}(\F/\F_{q})$.
  One requires the following condition to be fulfilled.
\begin{itemize}
\item[i)]
If $ s' \in (\bar{C}\setminus C)( \F)$ is ${\rm Gal}(
\F/\F_q)$-conjugate to $ s$, the conjugation which identifies  $I( s')$
and  $I( s)$ identifies  $V_{ s'}$ and $V_{ s}$. 
\end{itemize}
Let $V$ be an  irreducible lisse $\bar \Q_\ell $   sheaf of rank $r$ on $ C\otimes_{\F_q} \F$ such that   the set of  isomorphism classes of 
   restrictions  $\{V\otimes K_{ s} \}$ to $\Spec K_{s}$ is the set $\{V_{s}\}$ defined above with the condition i). 
      Then if  for a natural number $n\ge 1$, 
   $V$ is $F^n$ invariant, $V$ descends to a Weil sheaf on $C\otimes_{\F_q} \F_{q^n}$.  By
   Weil II, (1.3.3), ${\rm det}(V)$ is torsion. Thus by the dimension one case of Theorem~\ref{thm.finite1} the cardinality of the set of such $F^n$-invariant sheaves $V$ is finite.     
 \medskip
 
If  such a $V$ exists, then the set $\{V_{\bar s}\}$  satisfies automatically 
 \begin{itemize}
\item[ii)] For any $\ell' \neq p$, $\xi_{\ell'}$ acts trivially on $ \otimes_{ s\in (\bar{C}\setminus C)(\F)} {\rm det}(V_{ s})$.
\end{itemize}
 Indeed,  as ${\rm det}(V)$ is torsion,  a $p$ power ${\rm det}(V)^{p^N}$ has torsion $t$ prime to $p$, thus defines a class in $H^1(C \otimes_{\F_q} \F, \mu_t)$. The exactness of the localization sequence $H^1(C \otimes_{\F_{q}} \F, \mu_t) \xrightarrow{{\rm res}} \oplus_{ s\in (\bar C  \setminus  C)(\F)} \Z/t \xrightarrow{{\rm sum}} H^2(\bar{C}\otimes_{\F_q} \F, \mu_t)=\Z/t$ implies that the sum of the residues is $0$. This shows ii). 
 
\medskip

Furthermore, if  such a $V$ exists, then the set $\{V_{\bar s}\}$  satisfies automatically 
\begin{itemize}
 \item[iii)] The action of $\xi_{\ell'}$ on  $V_{ s}$ is quasi-unipotent for all $\ell'\neq p$ and all $ s\in  (\bar{C}\setminus C)(\F)$.
 \end{itemize}

 Indeed, this is Grothendieck's theorem,  see \cite[Appendix]{ST}. 
 \medskip 
 
 Given a set  $\{ V_{ s}\}$ for  all $ s\in (\bar{C}\setminus C)(\F)$, satisfying the conditions i), ii), iii),    Conjecture~\ref{conj} predicts a qualitative shape for the  cardinality of the $F^n$ invariants of the set $M$ of irreducible lisse $\bar \Q_\ell $   sheaves  on $ C\otimes_{\F_q} \F$ of rank $r$ 
 with $V\otimes K_{ s}$  isomorphic to $V_{ s}$.  
 
\medskip

\medskip

If $V$ is an element  of $M$, then $H^0(\bar C\otimes_{\F_q } \F, j_*\sE nd(V))=\bar
\Q_\ell$, spanned by the identity. Indeed, a global section is an endomorphism
$V\xrightarrow{f} V$ on $ C\otimes_{\F_q} \F$.  $f$ is  defined by an
endomorphism of the $\bar \Q_\ell$ vector space $V_{  a}$ which commutes with
the action of $\pi_1(\bar C,  a)$, where $ a\in C(\F)$ is a given closed
geometric point. Since this action is irreducible, the endomorphism is a
homothety.  We write $\sE nd(V)=\sE nd(V)^0\oplus \bar \Q_\ell$, where $\sE
nd(V)^0$ is the trace-free part,  thus 
$j_*\sE nd(V)=j_*\sE nd(V)^0\oplus \bar \Q_\ell$. Thus $H^0(\bar{ C}\otimes_{ \F_q} \F,
j_*\sE nd^0(V)) =0$. The cup-product 
\ga{}{   j_* \sE nd(V)) \times j_*\sE nd(V)\to j_*\bar \Q_\ell=\bar \Q_\ell  \notag}
obtained by composing endomorphisms and then taking the trace induces 
the perfect duality
\ga{dual}{    H^i(  \bar{C}\otimes_{ \F_q } \F, j_*\sE nd^0(V))\times H^{2-i}(   \bar{C}\otimes_{\F_q}\F, j_*\sE nd^0(V))\to H^2( \bar{C}\otimes_{\F_q}\F, \bar \Q_\ell) . } 
For $i=1$,  the bilinear form \eqref{dual} is symplectic.  We conclude that $H^2( \bar{C}\otimes_{\F_q} \F, j_*\sE nd^0(V))=0$ and that 
$H^1(  \bar{C}\otimes_{\F_q}\F, j_*\sE nd^0(V))$ is even dimensional. But ${\rm dim}  \ H^1  ( \bar{C}\otimes_{\F_q}\F, \bar{\Q}_\ell) =2g    $ thus 
$H^1(  \bar{C}\otimes_{\F_q}\F, j_*\sE nd(V))$ is even dimensional as well. 
We define 
\ga{}{2d={\rm dim}  \ H^1(  \bar{C}\otimes_{\F_q}\F, j_*\sE nd(V)).\notag}
\begin{conj}\label{conj}(Deligne's conjecture) \ \ 
\begin{itemize}
\item[i)] There are finitely many Weil numbers $a_i, b_j$ of weight between $ 0$  and $2d$  such that
\ga{}{    N(n)=\sum_i a_i^n-\sum_j b_j^n       \notag}
\item[ii)] If $M \neq \emptyset$, there is precisely one of the numbers $a_i, b_j$ of weight $2d$ and moreover, it is one of the $ a_i$ and is equal to $q^d$.
\end{itemize}

\end{conj}
An example where $M=\emptyset$ is given by $\bar{C}=\P^1$, $C$ is the complement of $3$ rational points $\{0,1,\infty\}$, the rank $r$ is $2$ and the $V_{s}$ are unipotent, so in particular, the Swan conductor at the $3$ points is $0$. Indeed, fixing $\ell'$, the  inertia groups $I( s)$ at the $3$ points, which depend on the choice of a base point,  can be chosen so the product over the $3$ points of the  $\xi_{\ell'}$ is equal to $1$. Thus  the set $\{V_{ s}, s=0,1,\infty\}$ is defined by  $3$ unipotent matrices $A_0,A_1,A_{\infty}$ in ${\rm GL}(2, \bar{\Q}_\ell)$ such that $A_0\cdot A_1\cdot A_{\infty}=1$.  Since $A_0\cdot A_1$ is then unipotent, $A_0$ and $A_1$, and thus $A_\infty$,  lie in the same Borel subgroup of $GL(2, \bar \Q_\ell)$. Thus 
the $3$ matrices have one common eigenvector. 
Since the tame fundamental group  is spanned by the images of $I( 0), I( 1), I(\infty)$,  a $\bar \Q_\ell$-sheaf of  rank $2$ with $V\otimes K_{s}$  isomorphic to $V_{ s}$ is not irreducible. Thus $M=\emptyset$.

\medskip

Two further examples are computed in \cite{DF}.  For the first case \cite[section~7]{DF},  $C=\P^1\setminus D$ where $D$ is a reduced degree $4$ divisor, with unipotent $V_{\bar s}$. The answer is $N(n)=q^n$. For the second case, 
 $C=\P^1\setminus D$ where $D$ is a reduced non-irreducible degree $3$ divisor with unipotent $V_{\bar s}$ with only one Jordan block (a condition which could be forced by the irreducibility condition for $V$). Then  $N(n)=q^n$ as well.

\appendix
\section{ } \label{app.A}

\noindent 
In this appendix we gather a few facts  on how to recognize through their
closed points affine schemes of finite type as subschemes of affine schemes not
necessarily of finite type.

\begin{prop} \label{HNS}
Let  $k$ be an algebraically closed  field, let $Y$ be an affine
$k$-scheme. Then 
the map $$Z\mapsto Z(k)$$ 
embeds
 the set of   reduced closed subschemes $Z\hookrightarrow Y$ of finite type into  the 
power set $\sP(Y(k))$.

\end{prop}
\begin{proof}
Choose a filtered direct system  $ B_\alpha \subset B=k(Y)$ of affine $k$-algbras (of finite type), such that
$B= \varinjlim_\alpha B_\alpha $. Set $Y_\alpha=\Spec B_\alpha$. Consider two closed subschemes
\ga{eq.twoclsub}{
Z_1 = \Spec B/I_1  \hookrightarrow Y, \quad Z_2 = \Spec B/I_2  \hookrightarrow Y}
  of finite type over $k$ such that $Z_1(k)=Z_2(k) \subset Y(k)$. After replacing the direct system $\alpha$ by a cofinal
subsystem we can assume that $B_\alpha \to B/I_1$ and $B_\alpha \to B/I_2$ are
surjective. Hilbert's Nullestellensatz for the closed subschemes $Z_1 \hookrightarrow
Y_\alpha$ and $Z_2 \to Y_\alpha$ implies $I_1 \cap B_\alpha = I_2 \cap B_\alpha$. So $I_1
=I_2$ and the closed subschemes \eqref{eq.twoclsub} agree.
\end{proof}

\begin{prop} \label{descent}
 Let $k$ be a characteristic $0$ field, let $K\supset k$ be an algebraically
closed field extension. Let $ Y$ be an affine scheme over $k$, and
$Z\hookrightarrow Y\otimes_k K$ be a closed embedding of an affine scheme of a
finite type. If the subset  $Z(K)$ of $Y(K)$ is invariant under the
automorphism group of $K$ over $k$, then there is a reduced closed
subscheme  $Z_0\hookrightarrow Y$ of finite type over $k$ such that
$$(Z\hookrightarrow Y\otimes_k K)= (Z_0\hookrightarrow Y)\otimes_k K.$$

\end{prop}
\begin{proof} 
Let $G={\rm Aut}(K/k)$, $B=k(Y)$, $Z= \Spec ( (B \otimes_k K) / I)$. The $G$-stability of
$Z(K)\subset Y(K)$ and  Proposition~\ref{HNS} imply that $I\subset B \otimes_k
K $ is stable under $G$. Then \cite[Sec.~V.10.4]{BourbakiAlg24} implies that $I_0=I^G
\subset B$ satisfies $I_0 \otimes_k K =I$. Set $Z_0= \Spec B/ I_0$.


\end{proof}

\begin{prop} \label{points}
 Let  $k$ be an algebraically closed field, 
 let
$\varphi: Z \to Y$ be an integral $k$-morphism of affine schemes, with $Z$ of
finite type over $k$. Then there is a uniquely defined reduced closed subscheme $X\hookrightarrow
Y$ of finite type over $k$
such that
$$ \varphi(Z(k))=X(k).$$
\end{prop}

\begin{proof}
Write $Y=\Spec B, \ Z=\Spec C$, for commutative $k$-algebras $B$, $C$ with $C$ of
finite type over $k$. Without loss of generality assume that $B$ and $C$ are reduced. There are finitely many elements of $C$ which span $C$ as a
$k$-algebra. They are integral over $B$. This defines finitely many
minimal polynomials, thus finitely many coefficients of those polynomials in
$B$. Thus there is an affine $k$-algebra of finite type $B_0\subset B$
containing them all. It follows that $C$ is finite over $B_0$. Choose a filtered inverse system 
$Y_\alpha=\Spec B_\alpha$ of affine $k$-schemes of finite type, such that $B_\alpha
\subset B$ and
$$Y=\Spec B= \varprojlim_{\alpha} Y_\alpha.$$ 
The morphisms $\varphi_\alpha: Z\xrightarrow{\varphi} Y\to Y_\alpha$ are all
finite. Let $X_\alpha=\Spec C_\alpha  \hookrightarrow Y_\alpha$ be the (reduced) image of
$\varphi_\alpha$. We obtain finite ring extensions $C_\alpha \subset C$. By Noether's
basis theorem the filtered direct system $C_\alpha $ stabilizes at some $\alpha_0$. Then
\[
X = \Spec C_{\alpha_0} = \varprojlim_\alpha \Spec C_{\alpha} \hookrightarrow Y
\] 
is of finite type over $k$ and satisfies  $\varphi(Z(k))=X(k)$.
\end{proof}

\begin{prop}\label{proinj}
 Let  $k$ be an algebraically closed field of characteristic $0$, let $Y$ be an affine
$k$-scheme, such that $Y=\Spec B= \varprojlim_{n} Y_n, \ n\in \N$ 
is the  projective limit of reduced affine schemes $Y_n$ of finite type. If the transition morphisms induce bijections 
$Y_{n+1}(k)\xrightarrow{\cong} Y_n(k)$ on closed points, then there is a $n_0\in \N$ such that $Y_n\to Y_{n_0}$ is an isomorphism for all $n\ge n_0$. In particular,  
$Y\to Y_{n_0}$ is an isomorphism as well.
\end{prop}
\begin{proof}
Applying Zariski's Main Theorem \cite[ Thm.4.4.3]{EGAIII}, one constructs inductively affine schemes of finite type 
$\bar{ Y}_n, \  \bar{ Y}_0=Y_0$, together 
with an open embedding $ Y_n \hookrightarrow \bar{Y}_n$, such that the transition morphisms $Y_{n+1}\to Y_n$ extend to finite transition morphisms
$\bar{ Y}_{n+1}\to \bar {Y}_n$. On the other hand, the assumption implies that the morphisms $Y_{n+1}\to Y_n$ are birational on every 
irreducible component. So the same property holds true for $\bar{Y}_{n+1}\to \bar{Y}_n$. 
 One thus has a factorization $\tilde{Y}_0\to \bar{Y}_n\to Y_0$ for all $n$, where $\tilde{Y}_0 \to Y_0$ the normalization morphism. Since 
$\tilde{Y}_0$ is of finite type, there is a $n_0$ such that $\bar{Y}_n\to \bar{Y}_{n_0}$ is an isomorphism for all $n\ge n_0$.  Thus the composite 
morphism $Y_n\to Y_{n_0}\to \bar{Y}_{n_0}$  is an open embedding  for all $n\ge n_0$, and thus $Y_{n+1}\to Y_n$ is an open embedding as well. 
 Since it induces a bijection on points, and the $Y_n$ are reduced, the transition morphisms $Y_{n+1}\to Y_n$ are isomorphisms for $n\ge n_0$.
\end{proof}
\begin{rmk}
If in Proposition~\ref{proinj}, one assumes in addition that the transition morphisms $Y_{n+1}\to Y_n$ are finite, then one does not need Zariski's Main Theorem to conclude.
\end{rmk}
\section{} \label{app.B}

\noindent
In the proof of Corollary \ref{corodetweight} we claim the existence of a curve
with certain properties. The Bertini argument
given in \cite[p.\ 201]{Laf} for the construction of such a curve is, as such,
not correct. We give a complete proof here relying on Hilbert
irreducibility instead of Bertini.

 Let $X$ be in $\Sm_{\F_q}$.

\begin{prop}\label{propcurveex}
For $V\in \dR_r(X)$ irreducible and a closed point $x\in X$, there is an irreducible smooth curve $C/ \F_q $ and
a morphism $\psi: C \to X$ such that
\begin{itemize}
\item $\psi^*(V)$ is irreducible,
\item $x$ is in the image of $\psi$.
\end{itemize}
\end{prop}

\smallskip

\begin{lem}
For an irreducible $\bar \Q_\ell$-\'etale sheaf $V$ on $X$
there is a connected \'etale covering  $X' \to X$ with the following property: \\
For a smooth irreducible curve $C/\F_q$ and a morphism
$\psi:C\to X$ the implication 
\[
 C \times_X X' \text{ irreducible }  \;\; \Longrightarrow \;\; \psi^*(V) \text{ irreducible }
\]
holds.
\end{lem}

\begin{proof}
Choose a finite normal extension $R$ of $\Z_\ell$ with maximal ideal $\mathfrak m\subset R$ such that $V$ is induced by a continuous representation
$$\rho:\pi_1(X) \to GL(R,r) .$$ Let $H_1$ be the kernel of $\pi_1(X) \to GL(R/\mathfrak m,r)$ and let $G$ be the image of $\rho$. The  subgroup
\[
H_2 = \bigcap_{\nu \in{\rm Hom}(H_1 , \Z /\ell )} \ker (\nu) 
\]
is open normal in $\pi_1(X)$ according to \cite[Th.\ Finitude]{SGA412}. Indeed
observe that $H_1/H_2 = H_1^{{\rm ab}}/\ell$ is Pontryagin dual to
$H^1_{{\rm \acute{e}t}}(X_{H_1},\Z / \ell)$,
where $X_{H_1}$ is the \'etale covering of $X$ associated to $H_1$.
Since the image of $H_1$ in $G$ is pro-$\ell$, and therefore pro-nilpotent, any
morphism of pro-finite groups $K\to \pi_1(X)$ satisfies:
\[
(K\to \pi_1(X) /H_2 \;\; \text{ surjective }) \Longrightarrow \;\; (K\to G \text{ surjective }).
\]
(Use \cite[Cor.\ I.6.3.4]{BourbakiAlg24}.)

Finally, let $X'\to X$ be the Galois covering corresponding to $H_2$.
\end{proof}

\begin{proof}[Proof of Proposition \ref{propcurveex}]
We can assume that $X$ is affine. By Proposition~\ref{weil.etale.gen} we can, after some
twist, assume that $V$ is \'etale.  Let $X'$ be as in the lemma. By Noether
normalization, e.g.\ \cite[Corollary~16.18]{Eis}, there is a finite generically
\'etale morphism $$f:X\to \A^d.$$
Let $U\subset \A^d$ be an open dense subscheme such that $f^{-1}(U) \to U$ is finite \'etale.
Let $y\in \A^d$ be the image of $x$.   Choose  a linear projection
$\pi: \A^d \to \A^1$ and set $z=\pi(y)$ and consider the map $ h:U\to \A^1$. By
definition, $U_{k(\A^1)}\subset \A^{d-1}_{k(\A^1)}$.  

 Let  $F=k(\Gamma) \supset k(\A^1)$ be a finite extension
such that $X'\otimes_{k(\A^1)} F$ is irreducible and the smooth curve $\Gamma\to
\A^1$ contains a closed point $z'$ with $k(z')=k(y)$. 

 It is easy to see
that there is an $\hat F$-point in $U_{k(\A^1)}$ which specializes to $y$. By
Hilbert irreducibility, see \cite[Cor.\ A.2]{Drinfeld}, we find an $F$-point
$u\in U_{k(\A^1)}$ which specializes to $y$ and such that $u$ does not split
in $X'\times_{\A^1} \Gamma$.

Let $v\in X$ be the unique point over $u$. By the going-down theorem \cite[Thm.\ V.2.4.3]{Bourbakicomalg} the closure $\overline{\{ v\}}$ contains $x$. Finally, we let $C$ be the normalization of $\overline{\{ v\}}$.
\end{proof}

\bibliographystyle{plain}

\begin{thebibliography}{99}
\bibitem{EGAII}
Grothendieck, A. {\it  Etude globale \'elémentaire de quelques classes de morphismes},
Publ. Math. I.H.\'E.S.  {\bf 8} (1961), 5--222.
\bibitem{EGAIII} Grothendieck, A. {\it \'Etude cohomologique des faisceaux coh\'erents}, EGA III, premi\`ere partie, Publ. Math. I.H.\'E.S. {\bf 11} (1961), 5--167.
\bibitem{SGA412}Deligne, P. {\it Cohomologie \'etale},
S\'eminaire de G\'eom\'etrie Alg\'ebrique du Bois-Marie SGA {\bf 4$
\frac{1}{2}$}. Avec la collaboration de J. F. Boutot, A. Grothendieck, L.
Illusie et J. L. Verdier. Lecture Notes in Mathematics, Vol. {\bf 569}.
Springer-Verlag, Berlin-New York, 1977. 
\bibitem{AS} Abbes, A., Saito, T.  {\it  Ramification and cleanliness,} Tohoku
Math. J. (2) {\bf 63} (2011), no. 4, 775--853.
\bibitem{BBD} Beilinson, A. A.; Bernstein, J.; Deligne, P. {\it
Faisceaux pervers}, 
Ast\'erisque, {\bf 100}, Soc. Math. France, Paris, 1982. 
\bibitem{BourbakiAlg24} Bourbaki, N. {\it El\'ements de math\'ematique. Alg\`ebre}. 
\bibitem{Bourbakicomalg}Bourbaki, N. {\it El\'ements de math\'ematique. Alg\`ebre commutative}.
\bibitem{CTSS} Colliot-Th\'el\`ene, J.-L., Sansuc, J.-J., Soul\'e, Ch.. {\it Quelques th\'eor\`emes de finitude en th\'eorie
des cycles alg\'ebriques}, C.R.A.S. {\bf 294} (1982), 749--752.
\bibitem{WeilII} Deligne, P.. {\it La conjecture de Weil II}, Inst. Hautes
\'Etudes Sci. Publ.  Publ. Math., 
no. {\bf 52} (1980), 137--252.
\bibitem{Delepsilon}
Deligne, P. {\it Les constantes des \'equations fonctionnelles des fonctions
$L$},  Modular functions of one variable, II (Proc. Internat. Summer School,
Univ. Antwerp, Antwerp, 1972),  pp. 501--597. Lecture Notes in Math., Vol.
{\bf 349}, Springer, Berlin, 1973.
\bibitem{DelMos} Deligne, P. {\it Un th\'eor\`eme de finitude pour la
monodromie}, in Discrete Groups in Geometry and Analysis,  Progress in Math.
{\bf 67} (1987), 1--19.

\bibitem{DelNumberField} Deligne, P. {\it   Finitude de l'extension de $\mathbb{Q}$ engendr\'ee par des 
traces de Frobenius, en caract\'eristique  finie},   Volume dedicated to the memory of
I.~M.~Gelfand,  Moscow Math. J. {\bf 12} (2012) no. 3.


\bibitem{DelFinitude} Deligne, P. {\it Letter to V. Drinfeld dated June 18,
2011}, 9 pages.

\bibitem{DF} Deligne, P., Flicker, Y. {\it  Counting local systems with
principal unipotent local monodromy}, preprint 2011, 63 pages, {
http://www.math.osu.edu/\~{}flicker.1}


\bibitem{Drinfeld} Drinfeld, V. 
{\it On a conjecture of Deligne},   Volume dedicated to the memory of I.~M.~Gelfand,
Moscow Math. J. {\bf 12} (2012) no. 3.

\bibitem{Eis} Eisenbud, D. {\it Commutative Algebra with a view towards
Algebraic Geometry}, Springer Verlag.   
\bibitem{ESV} 
Esnault, H., Srinivas, V., Viehweg, E. {\it The universal regular quotient of the Chow group of points on projective varieties}, 
 Invent. math. {\bf 135} (1999), 595--664.
 \bibitem{EK} Esnault, H., Kerz, M.  {\it  Notes on Deligne's letter to Drinfeld dated March 5, 2007},   Notes for the Forschungsseminar in Essen, summer 2011, 
 20 pages.
 \bibitem{Fa} Faltings, G. {\it Arakelov's theorem for abelian varieties }, Invent. math.
   {\bf 73} (1983), no. 3, 337--348.
\bibitem{FrJar}
Fried, M., Jarden, M.  {\it
Field arithmetic}, Springer-Verlag, Berlin, 2008.
\bibitem{Har} Hartshorne, R. {\it Algebraic Geometry}, Springer Verlag, Graduate Texts in Mathematics {\bf 52} ( 1977 ).
\bibitem{JPS}
Jacquet, H., Piatetski-Shapiro, I., Shalika, J. 
{\it  Conducteur des repr\'esentations du groupe lin\'eaire}, Math. Ann.
{\bf 256} (1981), no. 2, 199--214. 
\bibitem{Hat}
Hatcher, A. {\it
Algebraic topology},  Cambridge University Press, Cambridge, 2002.  
\bibitem{KS} Kato, K., Saito, S. {\it Global Class Field Theory of Arithmetic
Schemes}, Contemporary math.  {\bf 55} I (1986), 255--331.
\bibitem{Kedlaya} Kedlaya, K.
{\it More \'etale covers of affine spaces in positive characteristic}, 
J. Algebraic Geom. {\bf 14 }(2005), no. 1, 187--192. 
\bibitem{KeSch} Kerz, M.; Schmidt, A. {\it
On different notions of tameness in arithmetic geometry},
 Math. Ann.  {\bf 346}  (2010),  no. 3, 641-–668. 
\bibitem{Laf} Lafforgue, L.: {\it Chtoucas de Drinfeld et correspondance de
Langlands}, Invent. math.  {\bf 147} (2002), no. 1, 1--241.  
\bibitem{Laumon}
Laumon, G.. {\it Transformation de Fourier, constantes d'\'equations
fonctionnelles et conjecture de Weil}, Inst. Hautes \'Etudes Sci. Publ. Math. 
{\bf 65}  (1987), 131--210. 
\bibitem{LaumonDrinfeld} Laumon, G. {\it  Cohomology  of Drinfeld modular
varieties. Part II. Automorphic forms, trace formulas  and Langlands
correspondence. With an appendix by J.-L. Waldspurger. }  Cambridge Studies in
Advanced Mathematics {\bf 56}  (1997).
\bibitem{Raynaud}
Raynaud, M. {\it
Caract\'eristique d'Euler-Poincar\'e d'un faisceau et cohomologie des
vari\'et\'es ab\'eliennes},  S\'eminaire Bourbaki, Vol. 9, Exp. {\bf 286},
129--147, Soc. Math. France, Paris, 1995. 
\bibitem{Russell} Russell, H. {\it Generalized Albanese and its dual}. J. Math. Kyoto
  Univ. {\bf 48} (2008), no. 4, 907--949.
\bibitem{Serre}
Serre, J.-P. {\it
Corps locaux},
Deuxi\`eme \'edition. Publications de l'Universit\'e de Nancago,  VIII.
Hermann, Paris, 1968. 
\bibitem{ST} Serre, J.-P., Tate,  J.: Good Reduction of Abelian Varieties, Ann. of math. {\bf 88}  (3) (1968), 492--517.
\bibitem{Schm}
Schmidt, A. {\it
Some consequences of Wiesend's higher dimensional class field theory}, Appendix to: Class
field theory for arithmetic schemes, Math. Z. 256 (2007), no. 4, 717--729 by G. Wiesend. 
\bibitem{SchSp} 
Schmidt, A., Spiess, M.. {\it Singular homology of arithmetic schemes}, Algebra Number
Theory 1 (2007), no. 2,  183--222.
\bibitem{W}
Wiesend, G. {\it A construction of covers of arithmetic schemes}, J. Number
Theory {\bf 121} (2006), no. 1,  118--131.
\bibitem{W2} Wiesend, G.
{\it Class field theory for arithmetic schemes}, 
Math. Z. 256 (2007), no. 4, 717--729. 



\end{thebibliography}
\renewcommand\refname{References}

\end{document}